\documentclass{article}

\usepackage{blindtext}
\usepackage[utf8]{inputenc}

\usepackage{graphicx,fancyhdr,amsmath,amssymb,amsthm,subfig,url,hyperref}
\usepackage{algorithm}
\usepackage{algpseudocode}
\usepackage{booktabs}
\usepackage{float}
\restylefloat{table}
\usepackage[margin=1in]{geometry}
\usepackage{amsthm}
\usepackage{multirow}
\usepackage{enumitem}
\usepackage[normalem]{ulem}
\usepackage{accents}
\usepackage{thm-restate}
\usepackage{physics}
\usepackage[numbers]{natbib}

\usepackage{todonotes}

\newcommand{\studentname}{Oliver Hinder}
\newcommand{\suid}{ohinder}

\algnewcommand{\IfR}[1]{\State\algorithmicif\ #1\ \algorithmicthen}
\algnewcommand{\EndIfR}{\unskip\ \algorithmicend\ \algorithmicif}

\theoremstyle{plain}
\newtheorem{theorem}{Theorem}
\newtheorem{assumption}{Assumption}

\newtheorem{observation}[theorem]{Observation}

\newtheorem{corollary}[theorem]{Corollary}
\newtheorem{definition}[theorem]{Definition}

\newtheorem{fact}[theorem]{Fact}

\newcommand{\matStart}{\begin{pmatrix}}
\newcommand{\matEnd}{\end{pmatrix}}

\fancypagestyle{plain}{}
\graphicspath{{figures/}}

\def\R{R}

\def\MeritComp{\mathcal{\zeta}}

\def\Lag{\mathcal{L}}

\def\AlgTrust{Primal-dual-trust-region}
\newcommand{\callAlgTrust}[1]{\hyperref[AlgTrust]{\Call{\AlgTrust}{#1}}}
\def\AlgMain{non-convex-IPM}
\newcommand{\callAlgMain}[1]{\hyperref[AlgMain]{\Call{\AlgMain}{#1}}}

\usepackage{enumitem}

\begin{document}

\title{A one-phase interior point method for nonconvex optimization}
\author{Oliver Hinder, Yinyu Ye}

\algdef{SE}[SUBALG]{Indent}{EndIndent}{}{\algorithmicend\ }%
\algtext*{Indent}
\algtext*{EndIndent}

\maketitle

\newcommand{\fracBound}{fraction-to-boundary}
\newcommand{\kStart}{k_{\text{start}}}
\newcommand{\kEnd}{k_{\text{end}}}

\newcommand{\algorithmicbreak}{\textbf{break}}
\newcommand{\obj}{f}
\newcommand{\cons}{a}
\newcommand{\hess}{\grad^2}
\newcommand{\nvar}{n}
\newcommand{\ncon}{m}

\newcommand\figDir[1]{#1} 

\renewcommand{\vec}[1]{#1}

\newcommand{\eye}{I}
\newcommand{\ones}{\vec{e}}
\newcommand{\dir}[1]{\vec{d}_{\vec{#1}}}
\newcommand{\Dir}[1]{D_{\vec{#1}}}

\newcommand{\LDL}{LDL}
\newcommand{\LBL}{LBL}


\newcommand{\parNumCor}{j_{\max}}
\newcommand{\parNumCorValue}{2}



\newcommand{\parConRegularizer}{\beta_{1}}
\newcommand{\parConRegularizerValue}{10^{-4}}
\newcommand{\parConRegularizerInterval}{(0,1)}

\newcommand{\parComp}{\beta_{2}}
\newcommand{\parCompValue}{0.01}
\newcommand{\parCompInterval}{(0,1)}

\newcommand{\parCompAgg}{\beta_{3}}
\newcommand{\parCompAggValue}{0.02}
\newcommand{\parCompAggInterval}{(\parComp,1)}

\newcommand{\parObjReductFactor}{\beta_{4}}
\newcommand{\parObjReductFactorValue}{0.2}
\newcommand{\parObjReductFactorInterval}{(0,1)}

\newcommand{\parMinStableStepSize}{\beta_{5}}
\newcommand{\parMinStableStepSizeValue}{2^{-5}}
\newcommand{\parMinStableStepSizeInterval}{(0,1)}

\newcommand{\parKKTReductFactor}{\beta_{7}}
\newcommand{\parKKTReductFactorValue}{0.01}
\newcommand{\parKKTReductFactorInterval}{(0,1)}

\newcommand{\parBacktracking}{\beta_{6}}
\newcommand{\parBacktrackingValue}{0.5}
\newcommand{\parBacktrackingInterval}{(0,1)}

\newcommand{\parFracBoundaryExp}{\beta_{7}}
\newcommand{\parFracBoundaryExpValue}{0.5}
\newcommand{\parFracBoundaryExpInterval}{(0,1)}

\newcommand{\parAggProtect}{\beta_{8}}
\newcommand{\parAggProtectValue}{0.9}
\newcommand{\parAggProtectInterval}{(0.5,1)}

\newcommand{\parFracBoundary}{\Theta^{b}}
\newcommand{\parFracBoundaryValue}{0.1}
\newcommand{\parFracBoundaryInterval}{(0,1)}

\newcommand{\parFracBoundaryMax}{\Theta^{p}}
\newcommand{\parFracBoundaryMaxValueLinear}{0.1}
\newcommand{\parFracBoundaryMaxValueNL}{0.25}
\newcommand{\parFracBoundaryMaxInterval}{[\parFracBoundary_{i,i}, 1)}

\newcommand{\parInitialize}{\beta_{10}}
\newcommand{\parInitializeValue}{10^{-4}}
\newcommand{\parInitializeInterval}{(0,1)}

\newcommand{\parInitializeMin}{\beta_{11}}
\newcommand{\parInitializeMinValue}{10^{-2}}
\newcommand{\parInitializeMinInterval}{(0,\infty)}

\newcommand{\parInitializeMax}{\beta_{12}}
\newcommand{\parInitializeMaxValue}{10^{3}}
\newcommand{\parInitializeMaxInterval}{[\parInitializeMin,\infty)}

\newcommand{\parMuScale}{\mu_{\text{scale}}}
\newcommand{\parMuScaleValue}{1.0}
\newcommand{\parMuScaleInterval}{(0,\infty)}

\newcommand{\minStepFunc}{\theta}



\newcommand{\parDeltaMin}{\Delta_{\min}}
\newcommand{\parDeltaMinValue}{10^{-8}}
\newcommand{\parDeltaMinInterval}{(0,\infty)}

\newcommand{\parDeltaIncreaseFailure}{\Delta_{\text{inc}}}
\newcommand{\parDeltaIncreaseFailureValue}{8}
\newcommand{\parDeltaIncreaseFailureInterval}{(1,\infty)}

\newcommand{\parDeltaDecrease}{\Delta_{\text{dec}}}
\newcommand{\parDeltaDecreaseValue}{\pi}
\newcommand{\parDeltaDecreaseInterval}{(1,\infty)}

\newcommand{\parDeltaMax}{\Delta_{\max}}
\newcommand{\parDeltaMaxValue}{10^{50}}
\newcommand{\parDeltaMaxInterval}{(\parDeltaMin,\infty)}

\newcommand{\deltaPrev}{\delta_{\text{prev}}}
\newcommand{\parDeltaStart}{\Delta_{\text{start}}}

\newcommand{\parMaxItValue}{3,000}

\newcommand{\TOLopt}{\epsilon_{\textbf{opt}}}
\newcommand{\TOLoptValue}{10^{-6}}
\newcommand{\TOLoptInterval}{(0,\infty)}

\newcommand{\TOLinf}{\epsilon_{\textbf{inf}}}
\newcommand{\TOLinfValue}{10^{-6}}
\newcommand{\TOLinfInterval}{(0,1)}

\newcommand{\TOLinfOne}{\epsilon_{\textbf{far}}}
\newcommand{\TOLinfOneValue}{10^{-3}}
\newcommand{\TOLinfOneInterval}{(0,1)}

\newcommand{\TOLinfTwo}{\epsilon_{\textbf{inf}}}
\newcommand{\TOLinfTwoValue}{10^{-6}}
\newcommand{\TOLinfTwoInterval}{(0,1)}

\newcommand{\TOLall}{\epsilon}

\newcommand{\infeasFuncOne}{\infeasFunc_{\text{far}}}
\newcommand{\infeasFuncTwo}{\infeasFunc_{\text{inf}}}

\newcommand{\TOLunbounded}{\epsilon_{\textbf{unbd}}}
\newcommand{\TOLunboundedValue}{10^{-12}}
\newcommand{\TOLunboundedInterval}{(0,1)}


\newcommand{\History}{\mathbb{H}}
\newcommand{\conWeight}{\vec{w}}
\newcommand{\ConWeight}{W}
\newcommand{\zeroSet}{Z}
\newcommand{\nonzeroSet}{N}
\newcommand{\alphaMinAgg}{\alpha_{\min}}
\newcommand{\barrier}{\psi}
\newcommand{\regularizer}{r}
\renewcommand{\R}{\mathbb{R}}
\newcommand{\Schur}{\mathcal{M}}
\newcommand{\MatrixSchur}[2]{\mathcal{H}_{#1}(#2)}

\newcommand{\termination}{\eqref{terminate-kkt}, \eqref{terminate-primal-infeasible} or \eqref{terminate-dual-infeasible}}
\newcommand{\infeasFunc}{\Gamma}

\newcommand{\maxgrad}{G}

\newcommand{\status}{\textbf{status}}
\newcommand{\success}{\textsc{success}}
\newcommand{\failure}{\textsc{failure}}
\newcommand{\maxDelta}{\textsc{max-delta}}

\newcommand{\feasible}{\textbf{feasible}}

\newcommand{\meritKKT}{\mathbb{K}}

\newlength\myindent 
\setlength\myindent{6em} 
\newcommand\bindent{%
  \begingroup 
  \setlength{\itemindent}{\myindent} 
  \addtolength{\algorithmicindent}{\myindent} 
}
\newcommand\eindent{\endgroup} 

\newcommand{\simpleIPM}{Simplified-One-Phase-nonconvex-IPM}
\newcommand{\callSimpleIPM}{\Call{Simplified-One-Phase-nonconvex-IPM}}

\newcommand{\backtrackBlurb}{\emph{Perform a backtracking line search on the primal step $\alpha_{P}$.} Trial step sizes $\alpha_{P} \in \{\alpha^{\max}_{P}, \parBacktracking \alpha^{\max}_{P}, \parBacktracking^2 \alpha^{\max}_{P}, \dots \}$ computing the trial point $(\mu^{+}, x^{+}, s^{+}, y^{+})$ as described in \eqref{eq:primal-iterate-update} and \eqref{eq:alpha-D}. Terminate with $\status = \success$ and return the trial point the first time each of the following conditions hold:}

\begin{abstract}
The work of \citet*{wachter2000failure} suggests that infeasible-start interior point methods (IPMs) developed for linear programming cannot be adapted to nonlinear optimization without significant modification, i.e., using a two-phase or penalty method. We propose an IPM that, by careful initialization and updates of the slack variables, is guaranteed to find a first-order certificate of local infeasibility, local optimality or unboundedness of the (shifted) feasible region. Our proposed algorithm differs from other IPM methods for nonconvex programming because we reduce primal feasibility at the same rate as the barrier parameter. This gives an algorithm with more robust convergence properties and closely resembles successful algorithms from linear programming. We implement the algorithm and compare with IPOPT on a subset of CUTEst problems. Our algorithm requires a similar median number of iterations, but fails on only 9\%  of the problems compared with 16\% for IPOPT. Experiments on infeasible variants of the CUTEst problems indicate superior performance for detecting infeasibility. 

The code for our implementation can be found at \url{https://github.com/ohinder/OnePhase}.
\end{abstract}

\section{Introduction}

Consider the problem
\begin{subequations}\label{original-problem} 
\begin{flalign}
\min_{x \in \R^{\nvar}}{\obj(x)} \\
\cons(x) \le 0,
\end{flalign}
\end{subequations}
where the functions $a : \R^{\nvar} \rightarrow \R^{\ncon}$ and $f : \R^{\nvar} \rightarrow \R$ are twice differentiable and might be nonconvex. Examples of real-world problems in this framework include truss design, robot control, aircraft control, and aircraft design, e.g., the problems TRO11X3, ROBOT, AIRCRAFTA, AVION2 in the CUTEst test set \cite{gould2015cutest}. This paper develops an interior point method (IPM) for finding KKT points of \eqref{original-problem}, i.e., points such that
\begin{subequations}\label{kkt-of-original-problem} 
\begin{flalign}
\grad \obj(x) + \grad \cons(x)^T y &= 0 \\
y^T \cons(x) &= 0 \\
\cons(x) &\le 0, \quad y \ge 0.
\end{flalign}
\end{subequations}
IPMs were first developed by \citet*{karmarkar1984new} for linear programming. The idea for primal-dual IPMs originates with \citet*{megiddo1989pathways}. Initially, algorithms that required a feasible starting point were studied \cite{kojima1989primal,monteiro1989interior}. However, generally one is not given an initial point that is feasible. A naive solution to this issue is to move the constraints into the objective  by adding a large penalty for constraint violation (Big-M method) \cite{mcshane1989implementation}. A method to avoid the penalty approach, with a strong theoretical foundation for linear programming, is the homogeneous algorithm \cite{andersen1998computational,andersen1999homogeneous,ye1994nl}. This algorithm measures progress in terms of the KKT error, which may not monotonically decrease in the presence of nonconvexity\footnote{This occurs even in the one-dimensional unconstrained case, e.g., consider minimizing $f(x)=-9 x - 3 x^2 + x^4/4$ starting from zero; the gradient norm increases then decreases.}. It is therefore difficult to generalize the homogeneous algorithm to nonconvex optimization. An alternative to the homogeneous algorithm is the infeasible-start algorithm of Lustig \cite{lustig1990feasibility}, which has fewer numerical issues and a smaller iteration count than the big-M method of \cite{mcshane1989implementation}. 
Lustig's approach was further improved in the predictor-corrector algorithm of \citet*{mehrotra1992implementation}. This algorithm reduced complementarity, duality and primal feasibility at the same rate, using an adaptive heuristic. This class of methods was shown by \citet*{todd2003detecting} to converge to optimality or infeasibility certificates (of the primal or dual). 

The infeasible-start method for linear programming of \citet{lustig1990feasibility} naturally extends to nonlinear optimization \cite{kortanek1997infeasible},
and most interior point codes for nonconvex optimization are built upon these ideas \cite{byrd2006knitro,vanderbei1999loqo,wachter2006implementation}. However, \citet*{wachter2000failure} showed that for the problem
\begin{subequations}\label{failure-ex}
\begin{flalign}
\min { x }\\
x^2 - s_{1} &= -1 \\
x - s_{2} &= 1 \\
s_1, s_2 &\ge 0
\end{flalign}
\end{subequations}
a large class of infeasible-start algorithms fail to converge to a local optimum or an infeasibility certificate starting at any point with $x < 0$, $s_{1} > 0$ and $s_{2} > 0$. Following that paper, a flurry of research was published suggesting different methods for resolving this issue \cite{benson2004interior,chen2006interior,curtis2012penalty,gould2015interior,liu2004robust,wachter2006implementation}. The two main approaches can be split into penalty methods \cite{chen2006interior,curtis2012penalty,gould2015interior,liu2004robust} and two-phase algorithms \cite{wachter2006implementation}. 

Penalty methods move some measure of constraint violation into the objective. These methods require a penalty parameter that measures how much the constraint violation contributes to the objective. 
Penalty methods will converge only if the penalty parameter is sufficiently large. However, estimating this value is difficult: too small and the algorithm will not find a feasible solution; too big and the algorithm might be slow and suffer from numerical issues. Consequently, penalty methods tend to be slow \cite[Algorithm 1]{curtis2012penalty} or use complex schemes for dynamically updating the penalty parameter \cite[Algorithm 2]{curtis2012penalty}. 

The algorithm IPOPT is an example of a two-phase algorithm: it has a main phase and a feasibility restoration phase  \cite{wachter2006implementation}. The main phase searches simultaneously for optimality and feasibility using a classical infeasible-start method. The feasibility restoration phase aims to minimize primal infeasibility. It is called only when the main phase fails, e.g., the step size is small. It is well known that this approach has drawbacks. The algorithm has difficulties detecting infeasibility \cite[Table 15]{huang2016solution} and will fail if the feasibility restoration phase is called too close to the optimal solution. Some of these issues have been addressed by \citet*{nocedal2014interior}. 

Our main contribution is an infeasible-start interior point method for nonlinear programming that builds on the work for linear programming of \citet{lustig1990feasibility}, \citet{mehrotra1992implementation}, and \citet*{mizuno1993adaptive}. The algorithm avoids a big-M or a two-phase approach. Furthermore, our solution to the issue posed in example \eqref{failure-ex} is simple: we carefully initialize the slack variables and use nonlinear updates to carefully control the primal residual. Consequently, under general conditions we guarantee that our algorithm will converge to a local certificate of optimality, local infeasibility or unboundedness. Our algorithm has other desirable properties. Complementarity moves at the same rate as primal feasibility. This implies from the work of \citet*{haeser2017behavior} that if certain sufficient conditions for local optimality conditions hold, our approach guarantees that the dual multipliers sequence will remain bounded. In contrast, in methods that reduce the primal feasibility too quickly, such as IPOPT,  the dual multiplier sequence can be unbounded even for linear programs. We compare our solver with IPOPT on a subset of CUTEst problems. Our algorithm has a similar median number of iterations to IPOPT, but fails less often. Experiments on infeasible variants of the CUTEst problems indicate superior performance of our algorithm for detecting infeasibility.

The paper is structured as follows. Section~\ref{sec:simple-alg} describes a simple version of our proposed one-phase interior point algorithm and gives intuitive explanations for our choices. Section~\ref{sec:theory} focuses on convergence proofs. In particular, Section~\ref{sec:infeas-criteron-justify} justifies the choice of infeasibility criterion, and Section~\ref{sec:global-conv} provides convergence proofs for the algorithm described in Section~\ref{sec:simple-alg}. Section~\ref{sec:practical-alg} gives a practical version of the one-phase algorithm. Section~\ref{sec:numerical-results} presents numerical results on the CUTEst test set.

\paragraph{Notation} We use the variables $x$, $s$ and $y$ to denote the primal, slack and dual variables produced by the algorithm. The diagonal matrices $S$ and $Y$ are formed from the vectors $s$ and $y$ respectively. Given two vectors $u$ and $v$, $\min\{ u, v \}$ is vector corresponding to the element-wise minimum. The norm $\norm{ \cdot }$ denotes the Euclidean norm. The letter $\ones$ represents the vector of all ones. The set $\R_{++}$ denotes the set of strictly positive real numbers.

\section{A simple one-phase algorithm}\label{sec:simple-alg}

Consider the naive log barrier subproblems of the form
\begin{flalign} \label{naive-log-barrier}
\min_{x \in \R^{\nvar}}{\obj(x) - \mu \sum_i{ \log{(-a_i(x))} } }.
\end{flalign}
The idea is to solve a sequence of such subproblems with $\mu \rightarrow 0$ and $\mu > 0$. The log barrier transforms the non-differentiable original problem~\eqref{original-problem} into a twice differentiable function on which we can apply Newton's method. However, there are issues with this naive formulation: we are rarely given a feasible starting point and one would like to ensure that the primal variables remain bounded. To resolve these issues we consider shifted and slightly modified subproblems of the form

\begin{flalign*}
\min_{x \in \R^{\nvar}} \barrier_{\mu}(x) := \obj(x) - \mu  \sum_i{ \left( \parConRegularizer \cons_i(x) + \log \left( \mu \conWeight_i - \cons_i(x)  \right) \right)  }, %
\end{flalign*}
where $\parConRegularizer \in \parConRegularizerInterval$ is a constant with default value $\parConRegularizerValue$, $\conWeight \ge 0$ is a vector that remains fixed for all subproblems, and some $\mu > 0$ measures the size of the shift. The purpose of the term $\parConRegularizer \cons_i(x)$ is to ensure that $-(\parConRegularizer \cons_i(x) + \log \left( \mu \conWeight_i - \cons_i(x) \right) )$ remains bounded below. This prevents the primal iterates from unnecessarily diverging. We remark that this modification of the log barrier function is similar to previous works \cite[Section 3.7]{wachter2006implementation}.

Holistically, our technique consists of computing two types of direction: stabilization and aggressive directions. Both directions are computed from the same linear system with different right-hand sides. Aggressive directions are equivalent to affine scaling steps \cite{mehrotra1992implementation} as they apply a Newton step directly to the KKT system, ignoring the barrier parameter $\mu$. Aggressive steps aim to approach optimality and feasibility simultaneously. However, continuously taking aggressive steps may cause the algorithm to stall or fail to converge. To remedy this we have a stabilization step that keeps the primal feasibility the same, i.e., aims to reduce the log barrier objective until an approximate solution to the shifted log barrier problem is found. While this step has similar goals to the centering step of Mehrotra, there are distinct differences. The centering steps of Mehrotra move the iterates towards the central path while keeping the primal and dual feasibility fixed. Our stabilization steps only keep the primal feasibility fixed while reducing the log barrier objective. This technique of alternating stabilization and aggressive steps, is analogous to the alternating predictor-corrector techniques of \citet*[Algorithm~1]{mizuno1993adaptive}.

The IPM that we develop generates a sequence of iterates $(x^{k},s^k, y^k, \mu^k)$ that satisfy
\begin{subequations}\label{eq:barrier-primal-sequence-nice}
\begin{flalign}
(x^{k},s^k, y^k, \mu^k) &\in \R^{\nvar} \times \R_{++}^{\ncon} \times \R_{++}^{\ncon} \times  \R_{++} \label{eq-domain} \\
\frac{s^{k}_i y^{k}_i}{\mu^{k}} &\in [ \parComp, 1 / \parComp] ~~ \forall i \in \{ 1, \dots, m \} \label{eq:comp-slack} \\
\cons(x^{k}) + s^{k} &= \mu^k \conWeight,  \label{eq:primal-feasibility} 
\end{flalign} 
\end{subequations}
where $\conWeight \ge 0$ is a vector for which $\cons(x^{0}) + s^{0} = \mu^0 \conWeight$, and $\parComp \in (0,1)$ is an algorithmic parameter with default value $\parCompValue$. This set of equations implies the primal feasibility and complementarity are moved at the same rate. 

Furthermore, there is a subsequence of the iterates $\pi_{k}$ (i.e., those that satisfy the aggressive step criterion \eqref{agg-criteron}) such that %

\begin{flalign}
\frac{\| \grad_{x} \Lag_{\mu^{\pi_{k}}}(x^{\pi_{k}}, y^{\pi_{k}}) \|_{\infty}}{\mu^{\pi_{k}}(\| y^{\pi_{k}} \|_{\infty} + 1)} &\le c, \label{eq:dual-feas}
\end{flalign}
where $c > 0$ is some constant and $\Lag_{\mu} (x, y) := \obj(x) + (y - \mu \parConRegularizer e)^T \cons(x)$ is the modified Lagrangian function. Requiring \eqref{eq:barrier-primal-sequence-nice} and \eqref{eq:dual-feas} is common in practical linear programming implementations \cite{mehrotra1992implementation}. Note that \eqref{eq:barrier-primal-sequence-nice} and \eqref{eq:dual-feas} can be interpreted as a `central sequence'. This is weaker than the existence of a central path, a concept from convex optimization \cite{andersen1999homogeneous,megiddo1989pathways}. Unfortunately, in nonconvex optimization there may not exist a continuous central path.

Conditions \eqref{eq:barrier-primal-sequence-nice} and \eqref{eq:dual-feas} are desirable because they imply the dual multipliers are likely to be well-behaved. To be more precise, assume the subsequence satisfying \eqref{eq:barrier-primal-sequence-nice} and \eqref{eq:dual-feas} is converging to a feasible solution. If this solution satisfies certain sufficiency conditions for local optimality, then the dual variables remain bounded and strict complementarity holds. We refer to our paper \cite{haeser2017behavior} for further understanding of this issue. A consequence of this property is that we can re-write equality constraints as two inequalities while avoiding numerical issues that might arise if we did this using other solvers \cite{haeser2017behavior}.

Related to this property of the dual multipliers being well-behaved is that our algorithm is explicitly designed for inequality constraints only. Often primal feasibility is written as $a(x)=0$, $x \ge 0$ as used by Mehrotra (and many others in the IPM literature). If we took the dual of a linear program before applying our technique then our method would be working with the same problem as the typical method, because for us dual feasibility is $\grad \obj(x) + \grad \cons(x)^T y = 0, y \ge 0$. We believe working in this form is superior for nonlinear programming where there is no symmetry between the primal and dual, and the form $a(x) \le 0$ has many benefits that we discuss shortly. Using inequalities has the following advantages:
\begin{enumerate}
\item It enables us to generate sequences satisfying \eqref{eq:barrier-primal-sequence-nice} and \eqref{eq:dual-feas}.
\item It allows us to use a Cholesky factorization instead of the  \LBL{} factorization of \citet*{bunch1971direct}. See end of Section~\ref{sub:direction-computation} for further discussion.
\item We avoid the need for a second inertia modification parameter to ensure nonsingularity of the linear system, i.e., $\delta_{c}$ in equation~(13) in \cite{wachter2006implementation}.  Not using $\delta_{c}$ removes issues where large modifications to the linear system may not provide a descent direction for the constraint violation.
\end{enumerate}
Other IPMs that use only inequalities include \cite{curtis2012penalty,vanderbei1999loqo}.

\subsection{Direction computation}\label{sub:direction-computation}

We now state how we compute directions, whose derivation is deferred to Section~\ref{sec:dir-derivation}. 

Let
\begin{flalign}
b &= \begin{bmatrix}
b_{D} \\
b_{P} \\
b_{C}
\end{bmatrix} = \begin{bmatrix}
 \grad_{x} \Lag_{\gamma \mu}(x,y) \\
(1-  \gamma)  \mu \conWeight \\
Y s - \gamma \mu \ones 
\end{bmatrix} \label{def:b}
\end{flalign}
be the target change in the KKT residual error, where $\grad_{x} \Lag_{\gamma \mu}(x,y)$ denotes $\grad_{x} \Lag_{\bar{\mu}}(x,y)$ with $\bar{\mu} = \gamma \mu$. The scalar $\gamma \in [0,1]$ represents the target reduction in constraint violation and barrier parameter $\mu$, with $\gamma = 1$ corresponding to stabilization steps and $\gamma < 1$ to aggressive steps. (For the simple one-phase algorithm, $\gamma = 0$ for the aggressive steps). The point $(\mu, x, s, y)$ denotes the current iterate. 

To compute the direction for the $x$ variables we solve
\begin{flalign}\label{eq:Schur-complement-system}
(\Schur + \delta I)  \dir{x} = -\left( b_{D} + \grad \cons(x)^T S^{-1} \left( Y b_{P} - b_{C} \right) \right),
\end{flalign}
where $\delta > 0$ is chosen such that $\Schur + \delta I$ is positive definite (the choice of $\delta$ is specified in Section~\ref{sec:simple-alg}) and
\begin{flalign}\label{eq:Schur-matrix}
\Schur = \grad_{xx}^2 \Lag_{\mu} (x, y) + \grad \cons(x)^T Y S^{-1} \grad \cons(x).
\end{flalign}
We factorize $\Schur + \delta I$ using Cholesky decomposition. The directions for the dual and slack variables are then
\begin{subequations}\label{compute-ds-dy}
\begin{flalign}
\dir{y} &\gets  -S^{-1} Y (\grad \cons(x)  \dir{x} + b_{P} - Y^{-1} b_{C}), \label{compute-dy} \\
\dir{s} &\gets -(1 - \gamma) \mu \conWeight - \grad \cons(x)  \dir{x}.  \label{compute-ds}
\end{flalign}
\end{subequations}
We remark that the direction $\dir{s}$ is not used for updating the iterates because $s$ is updated using a nonlinear update (Section~\ref{sec:update-iterates}), but we define it for completeness.

\subsubsection{Derivation of direction choice}\label{sec:dir-derivation} 
We now explain how we choose our directions \eqref{eq:Schur-complement-system} and \eqref{compute-ds-dy}. 

In our algorithm the direction $\dir{x}$ for the $x$ variables is computed with the goal of being approximately equal to $\dir{x}^{*}$ defined by
\begin{flalign}\label{sophisticated-barrier-problem}
\dir{x}^{*} \in \arg \min_{\bar{\vec{d}}_{x} \in \R^{\nvar}} & \barrier_{\gamma \mu}(x + \bar{\vec{d}}_{x}) + \frac{\delta}{2} \|\bar{\vec{d}}_{x} \|^2 %
\end{flalign}
with $\barrier_{\gamma \mu}$ denoting $\barrier_{\bar{\mu}}$ with $\bar{\mu}  = \gamma \mu$. This notation is used for subscripts throughout, i.e., $\Lag_{\gamma \mu}$ denotes $\Lag_{\bar{\mu}}$ with $\bar{\mu}  = \gamma \mu$.

Primal IPMs \cite{fiacco1990nonlinear} apply Newton's method directly to system \eqref{sophisticated-barrier-problem}. However, they have inferior practical performance to primal-dual methods that apply Newton's method directly to the optimality conditions. To derive the primal-dual directions let us write the first-order optimality conditions
\begin{flalign*}
\grad_{x} \Lag_{\gamma \mu}(x + \dir{x}^{*}, y + \dir{y}^{*}) + \delta \dir{x}^{*} &=  0  \\
\cons(x + \dir{x}^{*}) + s + \dir{s}^{*} &= \gamma \mu \conWeight \\
(S + \Dir{s}^{*}) (y + \dir{y}^{*}) &= \gamma \mu \ones \\
s + \dir{s}^{*}, y + \dir{y}^{*} &\ge 0,
\end{flalign*} 
where $(x,s,y)$ is the current values for the primal, slack and dual variables, $(x + \dir{x}^{*},y + \dir{y}^{*},s + \dir{s}^{*})$ is the optimal solution to \eqref{sophisticated-barrier-problem}, and $(\dir{x}^{*},\dir{y}^{*},\dir{s}^{*})$ are the corresponding directions ($\Dir{s}^{*}$ is a diagonal matrix with entries $\dir{s}^{*}$). Thus,
\begin{flalign}\label{primal-dual-Newton-direction}
\mathcal{K}_{\delta} \dir{}= -b,
\end{flalign}
where
\begin{flalign}
\mathcal{K}_{\delta} = \begin{bmatrix}
 \grad_{xx}^2 \Lag_{\mu}(x, y) + \delta I  & \grad \cons(x)^T & 0  \\
\grad \cons(x) & 0 & I \\
0 & S & Y
\end{bmatrix} \text{ and } d = \begin{bmatrix}
\dir{x} \\
\dir{y} \\
\dir{s}
\end{bmatrix}. \label{def:K-delta} 
\end{flalign}
Eliminating $\dir{s}$ from \eqref{primal-dual-Newton-direction} yields the symmetric system
\begin{flalign}\label{eq:ldl-system}
 \begin{bmatrix}
 \grad_{xx}^2 \Lag_{\mu}(x,y) + \delta I  & \grad \cons(x)^T  \\
\grad \cons(x) & -Y^{-1} S \\
\end{bmatrix}
\begin{bmatrix}
\dir{x} \\
\dir{y}
\end{bmatrix} 
=
-\begin{bmatrix}
b_{D} \\
b_{P} - Y^{-1} b_{C}
\end{bmatrix}.
\end{flalign}

\noindent This is similar to the system typically factorized by nonlinear programming solvers using \LBL{} \cite{andersen1998computational,byrd2006knitro,vanderbei1999loqo,wachter2006implementation}. If the matrix $\grad_{xx}^2 \Lag_{\mu}(x,y) + \delta I$ is positive definite the whole matrix is quasidefinite, and in this case one can perform an \LDL{} factorization with a fixed pivot order \cite{gill1996stability,vanderbei1995symmetric}. However, if $\grad_{xx}^2 \Lag_{\mu}(x,y) + \delta I$ is not positive definite then \LBL{} \cite{bunch1971direct} must be used to guarantee factorability and may require excessive numerical pivoting. One way of avoiding using an \LBL{} factorization is to take the Schur complement of \eqref{eq:ldl-system}. For this system, there are two possible Schur complements. We use the term primal Schur complement to mean that the final system is in terms of the primal variables, whereas the dual Schur complement gives a system in the dual variables. 

Taking the primal Schur complement gives system \eqref{eq:Schur-complement-system}:
\begin{flalign*}
(\Schur + \delta I)  \dir{x} = -\left( b_{D} + \grad \cons(x)^T S^{-1} \left( Y b_{P} - b_{C} \right) \right).
\end{flalign*}
Equations \eqref{compute-ds-dy} also follow from \eqref{primal-dual-Newton-direction}.

Note that if $\Schur + \delta I$ is positive definite and $\gamma = 1$, then the right-hand side of \eqref{eq:Schur-complement-system} becomes $-\grad \barrier_{\mu}(x)$; therefore $\dir{x}$ is a descent direction for the function $\barrier_{\mu}(x)$. Consequently, we pick $\delta > 0$ such that $\Schur + \delta I$ is positive definite. Furthermore, note that if $Y = S^{-1} \mu$ then \eqref{eq:Schur-complement-system} reduces to
$$
(\hess \psi_{\mu}(x) + \delta I) \dir{x}  = - \grad \barrier_{\mu}(x);
$$
hence $\Schur$ should be interpreted as a primal-dual approximation of the Hessian of $\barrier_{\mu}$.

We emphasize we are forming the primal Schur complement, not the dual. This is a critical distinction for nonlinear programming because there are drawbacks to using the dual Schur complement. First, the matrix $\Schur$ could be positive definite but $\grad_{xx}^2 \Lag_{\mu}(x,y)$ could be negative definite, indefinite or even singular. Consequently, one might need to add an unnecessarily large $\delta$ to make $\grad_{xx}^2 \Lag_{\mu}(x,y) + \delta I$ positive definite and compute the direction. This could slow progress and prohibit superlinear convergence. Second, this method requires computing $(\grad_{xx}^2 \Lag_{\mu}(x,y) + \delta I)^{-1}  \grad \cons(x)^T$, which is expensive if there is even a moderate number of constraints. Finally, if $\grad_{xx}^2 \Lag_{\mu}(x,y)$ is not a diagonal matrix, then, usually the dual schur complement will usually be dense (similar issues occur for reduced Hessian methods  \cite{walterThesis1,walterThesis2}). In contrast, $\Schur$ is generally sparse. Furthermore, if $\Schur$ is not sparse it is likely that the Jacobian $\grad \cons(x)$ has a dense row. This, however, could be eliminated through row stretching of the original problem, as is done for columns in linear programming \cite{grcar2012matrix,lustig1991formulating,vanderbei1991splitting}. 

\subsection{Updating the iterates}\label{sec:update-iterates}

Suppose that we have computed direction $(\dir{x}, \dir{s}, \dir{y})$ with some $\gamma \in [0,1]$ using \eqref{eq:Schur-complement-system} and \eqref{compute-ds-dy}. We wish to construct a candidate $(\mu^{+}, x^{+}, s^{+}, y^{+})$ for the next iterate.
Given a primal step size $\alpha_{P} \in [0,1]$ and dual step size $\alpha_{D} \in [0,1]$ we update the iterates as follows:
\begin{subequations}\label{eq:iterate-update}
\begin{flalign}
\mu^{+} &\gets (1 - (1 - \gamma) \alpha_{P}) \mu \label{eq:muVarUpdate} \\
\vec{x}^{+} &\gets \vec{x} + \alpha_{P} \dir{x} \label{eq:xVarUpdate} \\
\vec{s}^{+} &\gets \mu^{+} \conWeight - \cons(\vec{x}^{+}) \label{eq:slackVarUpdate} \\ 
\vec{y}^{+} &\gets \vec{y} + \alpha_{D} \dir{y}. \label{eq:update-y} 
\end{flalign}
\end{subequations}
The slack variable update does not use $\dir{s}$. Instead, \eqref{eq:slackVarUpdate} is nonlinear and its purpose is to ensure that \eqref{eq:primal-feasibility} remains satisfied, so that we can control the rate of reduction of primal feasibility. In infeasible-start algorithms for linear programming \cite{lustig1990feasibility,mehrotra1992implementation} the variables updates are all linear, i.e., $\vec{s}^{+} \gets \vec{s} + \alpha_{P} \dir{s}$. However, if the function $a$ is linear, the slack variable update~\eqref{eq:slackVarUpdate} reduces to
$$
s^{+} = \mu^{+} \conWeight - \cons(x) - \alpha_{P} \grad \cons(x)  \dir{x} = \left(\mu \conWeight - \cons(x) \right) -  \alpha_{P}  \left( (1 - \gamma) \mu \conWeight + \grad \cons(x)  \dir{x} \right) = s + \alpha_{P} \dir{s}
$$
where the first equality uses \eqref{eq:slackVarUpdate} and linearity of $a$, the second uses \eqref{eq:muVarUpdate}, and the final uses \eqref{compute-ds}. Furthermore, as $\dir{x} \rightarrow 0$ the linear approximation $\cons(x) + \alpha_{P} \grad \cons(x)$ of $\cons(x + \alpha_{P} \dir{x})$ becomes very accurate and we have $s^{+} \rightarrow  s + \alpha_{P} \dir{s}$. Nonlinear updates for the slack variables have been used in other interior point methods \cite{andersen1998computational, curtis2012penalty}.

Finally, we only select steps that maintain \eqref{eq-domain} and \eqref{eq:comp-slack}, i.e., satisfy
\begin{subequations}\label{eq:comp-slack-plus}
\begin{flalign}
s^{+}, y^{+}, \mu^{+} &> 0 \\
\frac{s^{+}_i y^{+}_i}{\mu^{+}} &\in [ \parComp, 1 / \parComp] ~~ \forall i \in \{ 1, \dots, m \}. 
\end{flalign}
\end{subequations}

\subsection{Termination criterion}

Define the function
$$
\sigma (y) := \frac{100}{\max\{ 100, \| y \|_{\infty} \}}
$$
as a scaling factor based on the size of the dual variables. This scaling factor is related to $s_{d}$ and $s_{c}$ in the IPOPT implementation paper \cite{wachter2006implementation}. We use $\sigma(y)$ in the local optimality termination criterion \eqref{terminate-kkt} because there may be numerical issues reducing the unscaled dual feasibility if the dual multipliers become large. In particular, the first-order optimality termination criterion we use is
\begin{subequations}\label{terminate-kkt}
\begin{flalign}
\sigma (y) \| \grad \Lag_0(x, y) \|_{\infty} &\le  \TOLopt  \\
\sigma (y) \| S y \|_{\infty} &\le \TOLopt  \\
\| \cons(x) + s \|_{\infty} &\le \TOLopt,
\end{flalign}
where $\TOLopt \in (0,1)$ with a default value of $\TOLoptValue$.
\end{subequations}
The first-order local primal infeasibility termination criterion is given by
\begin{subequations}\label{terminate-primal-infeasible}
\begin{flalign}
a(x)^T y &> 0 \\
\infeasFuncOne (\mu,x,s,y) &\le \TOLinfOne \\
\infeasFuncTwo (\mu,x,s,y) &\le \TOLinfTwo, 
\end{flalign}
\end{subequations}
where
\begin{flalign*}
\infeasFuncOne (x,y) &:= \frac{\| \grad  a(x)^T y \|_{1}}{ a(x)^T y } \\
\infeasFuncTwo (x,s,y) &:= \frac{\| \grad  a(x)^T y \|_{1} + s^T y}{ \| y \|_{1} },
\end{flalign*}
and $\TOLinfOne, \TOLinfTwo \in (0,1)$ with default values of $\TOLinfOneValue$ and $\TOLinfTwoValue$ respectively.
We remark that if we find a point with $\infeasFuncTwo (x,s,y) = 0$ then we have found a stationary point to a weighted $L_{\infty}$ infeasibility measure.
For a more thorough justification of this choice for the infeasibility termination criterion, see Section~\ref{sec:infeas-criteron-justify}.

The unboundedness termination criterion is given by
\begin{flalign}\label{terminate-dual-infeasible}
\frac{1}{\| x \|_{\infty}}\ge 1/\TOLunbounded,
\end{flalign}
where $\TOLunbounded \in (0,1)$ with default value $\TOLunboundedValue$. Note that since we require all the iterates to maintain $a(x) \le \mu \conWeight \le \mu^{0} \conWeight$ satisfying, the unboundedness termination criterion strongly indicates that the set $\{ x \in \R^{n} : a(x) \le  \conWeight \mu  \}$ is unbounded.

\subsection{The algorithm}\label{sec:simple-alg-details}

Before we outline our algorithm, we need to define the switching condition for choosing an aggressive step instead of a stabilization step. The condition is

\begin{subequations}\label{agg-criteron}
\begin{flalign}
\sigma (y) \| \grad \Lag_{\mu}(x, y) \|_{\infty} &\le \mu   \label{agg-criteron-opt} \\
\| \grad \Lag_{\mu}(x, y) \|_{1} &\le \| \grad f(x) - \parConRegularizer \mu e^T \grad a(x) \|_{1} +  s^T y \label{agg-criteron-farkas} \\
 \frac{s_i y_i}{\mu} &\in  [ \parCompAgg, 1 / \parCompAgg ] ~~ \forall i \in \{ 1, \dots, m \}, \label{agg-criteron-buffer} 
\end{flalign}
\end{subequations}
where the parameter $\parCompAgg \in (\parComp, 1)$ has a default value of $\parCompAggValue$. The purpose of \eqref{agg-criteron-opt} is to ensure that we have approximately solved the shifted log barrier problem and guarantees that this subsequence of iterates satisfies \eqref{eq:dual-feas}. Equation~\eqref{agg-criteron-farkas} helps ensure (as we show in Section~\ref{sec:global-conv}) that if the dual variables are diverging rapidly then the infeasibility termination criterion is met. Finally, equation~\eqref{agg-criteron-buffer} with $\parCompAgg > \parComp$ ensures we have a buffer such that we can still satisfy \eqref{eq:comp-slack} when we take an aggressive step. 

Algorithm~\ref{simple-one-phase} formally outlines our one-phase interior point method. It does not include the details for the aggressive or stabilization steps that are given in Algorithm~\ref{alg:simple-agg-step} and \ref{alg:simple-stb-step} respectively. Since Algorithm~\ref{simple-one-phase} maintains $\cons(x) + s = \conWeight \mu$ for each iterate, it requires the starting point to satisfy
$$
\cons(x^{0}) + s^{0} = \conWeight \mu^{0},
$$
with $\conWeight \ge 0$ and $\mu^{0} > 0$. For any fixed $x^{0}$ one can always pick sufficiently large $\conWeight$ and $\mu^{0}$ such that $\mu^{0} \conWeight > \cons(x^{0})$, and setting $s^{0} \gets  \mu^{0}  \conWeight- \cons(x^{0})$ meets our requirements.

\newcommand{\superlinear}[1]{}

\begin{algorithm}[H]
\textbf{Input:} a initial point $x^{0}$, vector $\conWeight \ge 0$, and variables $\mu^0, s^{0}, y^{0} > 0$  such that $\cons(x^{0}) + s^{0} = \conWeight \mu^{0}$ and equation~\eqref{eq:comp-slack} with $k=0$ is satisfied. Termination tolerances $\TOLopt \in \TOLoptInterval$, $\TOLinfOne \in \TOLinfOneInterval$, $\TOLinfTwo \in \TOLinfTwoInterval$ and $\TOLunbounded \in \TOLunboundedInterval$. \\
\textbf{Output:} a point $(\mu^k, x^k, s^k, y^k)$ that satisfies at least one of the inequalities \termination{} \\

For $k \in \{1, \dots, \infty\}$
\begin{enumerate}[label*=A.{\arabic*}]
\superlinear{ \item \emph{Increase time since last significant improvement.} Set $t \gets t + 1$.}
\item \label{simple-first-step} Set $(\mu,x,s,y) \gets (\mu^{k-1},x^{k-1},s^{k-1},y^{k-1})$.
\item \emph{Check termination criterion}. If \termination{} is satisfied then terminate the algorithm.
\item \label{alg-simple-delta-min} Form the matrix $\Schur$ using \eqref{eq:Schur-matrix}. Set $\delta_{\min} \gets \max\{0, \mu-2 \lambda_{\min}(\mathcal{M}) \}$.
\item If \eqref{agg-criteron}\superlinear{ or $t = 2$}  is satisfied go to line~\ref{simple-agg-step}, \superlinear{$t = 1$ go to stable superlinear}  otherwise go to line~\ref{simple-stb-step}.
\item\label{simple-agg-step} \emph{Take an aggressive step}. 
\begin{enumerate}[label*=.{\arabic*}]
\item If $\delta_{\min} > 0$ go to line~\ref{simple-agg-large-delta}.
\item Set $\delta \gets 0$. Run Algorithm~\ref{alg:simple-agg-step}. If it terminates with $\status = \success$, then set $(\mu^k,x^k,s^k,y^k) \gets (\mu^{+},x^{+},s^{+},y^{+})$ and go to line~\ref{simple-first-step}.
 \superlinear{\item If $\norm{ \grad \Lag(x^{+},y^{+}) } \le \mu /  2$ and $\mu^{+} \le \mu/2$, then set $t \gets 0$.}
\item \label{simple-agg-large-delta} Run Algorithm~\ref{alg:simple-agg-step} with sufficiently large $\delta$ such that the algorithm terminates with $\status = \success$.  Set $(\mu^k,x^k,s^k,y^k) \gets (\mu^{+},x^{+},s^{+},y^{+})$ and go to line~\ref{simple-first-step}.
\end{enumerate}
\item \label{simple-stb-step} \emph{Take a stabilization step}.  Set $\delta \gets \delta_{\min}$. Run Algorithm~\ref{alg:simple-stb-step}, set $(\mu^k,x^k,s^k,y^k) \gets (\mu^{+},x^{+},s^{+},y^{+})$. Go to line~\ref{simple-first-step}.
\end{enumerate}
\caption{A simplified one-phase algorithm}\label{simple-one-phase}
\end{algorithm}

\subsubsection{Aggressive steps}\label{sec:simple-agg-step}

The goal of the aggressive steps is to approach optimality and feasibility simultaneously. For this reason, we set $\gamma = 0$ for the aggressive direction computation. We require a minimum step size $\alpha_{P}$ of
 \begin{flalign}\label{simple-min-step-size-aggresssive}
\minStepFunc(\mu, s) := \min\left\{ 1/2, \frac{(\parCompAgg - \parComp)}{2  \parCompAgg \mu }  \min_{i \in \{ 1, ..., \ncon \} : \conWeight_i > 0}{\frac{s_i}{w_i}} \right\}.
\end{flalign}
If $\alpha_{P} < \minStepFunc(\mu, s)$, we declare the step a failure (see line~\ref{line:too-small-step} of Algorithm~\ref{alg:simple-agg-step}). However, on line~\ref{simple-agg-large-delta} of Algorithm~\ref{simple-one-phase}, we choose a sufficiently large $\delta$ such that $\alpha_{P}$ selected by Algorithm~\ref{alg:simple-agg-step} satisfies $\alpha_{P} \ge \minStepFunc(\mu, s)$; such a $\delta$ must exist by Lemma~\ref{lemma:agg-succeeds}. Furthermore, the primal aggressive step size $\alpha_{P}$ can be at most $ \min\{1/2, \mu \}$. This ensures line~\ref{simple-agg-select-alpha-P} is well-defined, i.e., the set of possible primal step sizes is closed.

\begin{algorithm}[H]
\textbf{Input:}  $\delta > 0$, the current point $(\mu, x, s, y)$ and the matrix $\Schur$.  \\
\textbf{Output:} A new point $(\mu^{+}, x^{+}, s^{+}, y^{+})$ and a $\status$.
\begin{enumerate}[label*=A.{\arabic*}]
\item Compute vector $b$ at point $(\mu, x, s, y)$ via \eqref{def:b} with $\gamma = 0$.
\item Compute direction $(\dir{x}, \dir{s}, \dir{y})$ via  \eqref{eq:Schur-complement-system} and \eqref{compute-ds-dy}.
\item \label{simple-agg-select-alpha-P} Select the largest $\alpha_{P} \in \left[ 0, \min\{1/2, \mu \} \right]$ such that the iterates $(\mu^{+}, x^{+}, s^{+}, y^{+})$ computed by \eqref{eq:iterate-update} satisfy \eqref{eq:comp-slack-plus} for some $\alpha_{D} \in [0,1]$. 
\item With the previous $\alpha_{P}$ fixed, select the largest $\alpha_{D}$ such that the iterates $(\mu^{+}, x^{+}, s^{+}, y^{+})$ computed by \eqref{eq:iterate-update} satisfy \eqref{eq:comp-slack-plus}.
\item \label{line:too-small-step} If $\alpha_{P} \ge \minStepFunc(\mu, s)$  then $\status =\success$; otherwise $\status =\failure$.
\end{enumerate}
\caption{Simplified aggressive step}\label{alg:simple-agg-step}
\end{algorithm}

\subsubsection{Stabilization steps}\label{sec:simple-stable}
The goal of stabilization steps (Algorithm~\ref{alg:simple-stb-step}) is to reduce the log barrier function $\barrier_{\mu}$. For this purpose we set $\gamma = 1$ during the direction computation. However, the log barrier function $\barrier_{\mu}$ does not measure anything with respect to the dual iterates. This might impede performance if $\| S y - \mu \ones \|_{\infty}$ is large but $\| \grad \barrier_{\mu}(x) \|$ is small. In this case, taking a large step might reduce the complementarity significantly, even though the barrier function increases slightly. Therefore we add a complementarity measure to the barrier function to create an augmented log barrier function:
\begin{flalign*}
\phi_{\mu}(x, s, y) := \barrier_{\mu}(x) + \frac{\| S y - \mu \ones \|_{\infty}^3}{\mu^2}.
\end{flalign*}
We say that the candidate iterates $(x^{+}, s^{+}, y^{+})$ have made sufficient progress on $\phi_{\mu}$ over the current iterate $(x, s, y)$ if
\begin{flalign}\label{eq:phi-sufficient-progress}
\phi_{\mu}(x^{+}, s^{+}, y^{+}) \le \phi_{\mu}(x, s, y) +  \alpha_{P} \parObjReductFactor \left( \frac{1}{2} \left( \grad \psi_{\mu}(x)^T  \dir{x} - \frac{\delta}{2} \alpha_{P} \norm{ \dir{x}}^2 \right) -  \frac{\| S y - \mu \ones \|_{\infty}^3}{\mu^2}  \right),
\end{flalign}
where $\parObjReductFactor \in (0,1)$ is a parameter with default value $\parObjReductFactorValue$.

\begin{algorithm}[H]
\textbf{Input:} Some $\delta > 0$, the current point $(\mu, x, s, y)$ and the matrix $\Schur$.  \\
\textbf{Output:} A new point $(\mu^{+}, x^{+}, s^{+}, y^{+})$.
\begin{enumerate}[label*=A.{\arabic*}]
\item Compute the vector $b$ at the point $(\mu, x, s, y)$ via \eqref{def:b} with $\gamma = 1$.
\item Compute direction $(\dir{x}, \dir{s}, \dir{y})$ via  \eqref{eq:Schur-complement-system} and \eqref{compute-ds-dy}.
\item Pick the largest $\alpha_{P} = \alpha_{D} \in [0,1]$ such that $(\mu^{+},x^{+}, s^{+}, y^{+})$ computed by \eqref{eq:iterate-update} satisfies \eqref{eq:phi-sufficient-progress} and \eqref{eq:comp-slack-plus}.
\end{enumerate}
\caption{Simplified stabilization step}\label{alg:simple-stb-step}
\end{algorithm}

\section{Theoretical justification}\label{sec:theory}

The goal of this section is to provide some simple theoretical justification for our algorithm. Section~\ref{sec:infeas-criteron-justify} justifies the infeasibility termination criterion. Section~\ref{sec:global-conv} proves that the Algorithm~\ref{simple-one-phase} eventually terminates. 

\subsection{Derivation of primal infeasibility termination criterion} \label{sec:infeas-criteron-justify}

Here we justify our choice of local infeasibility termination criterion by showing that it corresponds to a stationary measure for the infeasibility with respect to a weighted $L_{\infty}$ norm. We also prove that when the problem is convex our criterion certifies global infeasibility.

Consider the optimization problem
\begin{subequations}\label{infeasible-problem}
\begin{flalign}
\min_{x} \max_{i : \conWeight_i > 0}{ \frac{a_i(x)}{ w_i }  }  \\
\text{s.t. } a_i(x) \le 0, \forall i \text{ s.t. } \conWeight_i = 0,
\end{flalign}
\end{subequations}
for some non-negative vector $\conWeight$. For example, a natural choice of $\conWeight$ is $\conWeight_i = 0$ for variable bounds and $\conWeight_i = 1$ for all other constraints. This results in minimizing the $L_{\infty}$ norm of the constraint violation subject to variable bounds. Note that \eqref{infeasible-problem} is equivalent to the optimization problem
\begin{subequations}\label{feas-problem}
\begin{flalign}
\min_{x,s,\mu} \mu \\
\text{s.t. } \cons(x) + s = \mu \conWeight \\
s, \mu \ge 0.
\end{flalign}
\end{subequations}

The KKT conditions for \eqref{feas-problem} are
\begin{flalign*}
\cons(x) + s &= \mu \conWeight  \\
\grad \cons(x)^T \tilde{y}  &= 0 \\
\conWeight^T \tilde{y}  + \tau &= 1 \\
\tau \mu &= 0  \\
\tilde{y}^T s &= 0 \\
s, \mu, \tilde{y}, \tau &\ge 0.
\end{flalign*}
Note that if the sequence $(\mu^k,x^k,s^k,y^k)$ generated by our algorithm satisfies
\begin{flalign*}
\cons(x^k) + s^k = \mu^k \conWeight  \\
\infeasFuncTwo (x^k,s^k,y^k) \rightarrow 0
\end{flalign*}
then with $\tilde{y}^k = \frac{y^k}{\conWeight^T y^k}$, $\tau^k = 0$, the KKT residual for problem~\eqref{infeasible-problem} of the sequence $(x^k, s^k, \tilde{y}^k, \tau^k)$ tends to zero. However, this is a poor measure of infeasibility because $\infeasFuncTwo (x^k,s^k,y^k) \rightarrow 0$ may also be satisfied if the algorithm is converging to a feasible solution (i.e., if the norm of the dual multipliers tends to infinity). For this reason, our infeasibility criterion~\eqref{terminate-primal-infeasible} includes $\infeasFuncOne$ to help avoid declaring a problem infeasible when the algorithm is in fact converging towards a feasible solution. To make this more precise, we include the following trivial observation.

\begin{observation}
Assume the constraint function $a : \R^{\nvar} \rightarrow \R^{\ncon}$ is differentiable and convex, and that some minimizer $(\mu^{*}, x^{*})$ of \eqref{feas-problem} satisfies $\| x - x^{*} \|_{\infty} \le R$ for some $R \in (0,\infty)$. Suppose also that at some point $(x,s,y)$ with $s,y \ge 0$, $\infeasFuncOne(x,s,y) \le 1 / (2R)$
and $y^T a(x) > 0$. Then the system $\cons(x) \le 0$ has no feasible solution.
\end{observation}
\begin{proof}
We have
\begin{flalign*}
y^T \cons(x^{*}) &\ge y^T \left( \cons(x) + \grad a (x) (x^{*} - x) \right) \ge y^T \cons(x) (1 - R \times \infeasFuncOne(x,s,y) )   \ge y^T \cons(x) / 2 > 0,
\end{flalign*}
where the first inequality holds via convexity and $y \ge 0$, the second by definition of $\infeasFuncOne(x,s,y) := \frac{\| \grad  a(x)^T y \|_{1}}{ a(x)^T y }$, the third and fourth by assumption. Since $y \ge 0$, we deduce that $a_i(x^{*}) > 0$ for some $i$.
\end{proof}

Finally, we remark that if wish to find a stationary point with respect to a different measure of constraint violation $v(z)$ we can apply our solver to the problem
\begin{flalign*}
\min f(z) \\
a(x) \le z \\
z \ge 0 \\
v(z) \le 0,
\end{flalign*}
starting with $\conWeight_{k} = 0$ for $k = 1, \dots \ncon + \nvar$, $z^{0} = \max\{ a(x^0), e \}$, $\mu^0 > 0$ and $\conWeight_{\ncon + \nvar+1} = \frac{1 + v(z^0)}{\mu^0} > 0$. With this choice of $\conWeight$ one can see from \eqref{infeasible-problem} that if we do not find a feasible solution then we automatically find a solution to:
\begin{flalign*}
\min v(z)\\
a(x) \le z \\
z \ge 0.
\end{flalign*}
For example, if $v(z) = e^T z$ then we minimize the L1-norm of the constraint violation. 

\subsection{Global convergence proofs for Algorithm~\ref{simple-one-phase}}\label{sec:global-conv}

We now give a global convergence proof for Algorithm~\ref{simple-one-phase} as stated in Theorem~\ref{thm:global-convergence} in Section~\ref{subsec:main-result}. Since the proofs are mostly mechanical, we defer most of the proofs to Appendix~\ref{app:global-conv}. Our results hold under assumption~\ref{assume:diff} and \ref{assume:parameters}. 
\begin{assumption}\label{assume:diff}
Assume the functions $f : \R^{\nvar} \rightarrow \R$ and $a : \R^{\nvar} \rightarrow \R^{\ncon}$ are twice differentiable on $\R^{\nvar}$.
\end{assumption}

\begin{assumption}\label{assume:parameters}
The algorithm parameters satisfy $\parConRegularizer \in \parConRegularizerInterval$, $\parComp \in \parCompInterval$ and $\parCompAgg \in \parCompAggInterval$. The tolerances satisfy $\TOLopt \in \TOLoptInterval$, $\TOLinfOne \in \TOLinfTwoInterval$, $\TOLinfTwo \in \TOLinfTwoInterval$ and $\TOLunbounded \in \TOLunboundedInterval$. The vector $\conWeight \ge 0$.
\end{assumption}

Recall that the iterates of our algorithm satisfy \eqref{eq:barrier-primal-sequence-nice}, i.e.,
\begin{subequations}\label{restate:eq:barrier-primal-sequence-nice}
\begin{flalign}
(x,s, y, \mu) &\in \R^{\nvar} \times \R_{++}^{\ncon} \times \R_{++}^{\ncon} \times  \R_{++} \label{restate:eq-domain} \\
\frac{s_i y_i}{\mu} &\in [ \parComp, 1 / \parComp] ~~ \forall i \in \{ 1, \dots, m \} \label{restate:eq:comp-slack} \\
\cons(x) + s &= \mu \conWeight.  \label{restate:eq:primal-feasibility} 
\end{flalign} 
\end{subequations}

\subsubsection{Convergence of aggressive steps}
Here we show that after a finite number of aggressive steps, Algorithm~\ref{simple-one-phase} converges. Lemma~\ref{lemma:agg-succeeds} uses the fact that for large enough $\delta$ the slack variables can be absorbed to reduce the barrier parameter $\mu$. 

First, we show that for sufficiently large $\delta$ Algorithm~\ref{alg:simple-agg-step} succeeds. Therefore line~\ref{simple-agg-large-delta} of Algorithm~\ref{simple-one-phase} is well-defined. 

\begin{restatable}{lemma}{lemAggSucceeds}\label{lemma:agg-succeeds}
Suppose assumptions~\ref{assume:diff} and \ref{assume:parameters} hold. If $(\mu, x, s, y)$ satisfies \eqref{restate:eq:barrier-primal-sequence-nice} and the criterion for an aggressive step \eqref{agg-criteron}, then there exists some $\bar{\delta}$ such that for all $\delta > \bar{\delta}$ Algorithm~\ref{alg:simple-agg-step} returns $\status = \success$.
\end{restatable}

The proof of Lemma~\ref{lemma:agg-succeeds} is given in Section~\ref{sec:lemma:agg-succeeds}. The next Lemma demonstrates that the term $y^T \conWeight$ remains bounded for points generated by our algorithm satisfying the aggressive step criterion.

\begin{restatable}{lemma}{lemYWbounded}\label{lem:yw-bounded}
Suppose assumptions~\ref{assume:diff} and \ref{assume:parameters} hold. Then $\conWeight^T y$ is bounded above for all points $(\mu, x, s, y)$ that satisfy \eqref{eq:barrier-primal-sequence-nice} and the following.
\begin{enumerate}
\item The criterion for an aggressive step \eqref{agg-criteron}.
\item Neither the infeasibility termination criterion~\eqref{terminate-primal-infeasible} nor the unboundness criterion~\eqref{terminate-dual-infeasible}.
\end{enumerate}
\end{restatable}

The proof of Lemma~\ref{lem:yw-bounded} is given in Section~\ref{sub:lem:yw-bounded}.  Lemma~\ref{lem:yw-bounded} shows $\conWeight^T y$ is bounded above. It follows that the slack variables $s_i$ for $\conWeight_i > 0$ are bounded away from zero. This enables us to lower bound the minimum aggressive step size $\minStepFunc(\mu, s)$, leading to Corollary~\ref{coro:agg-finite}.

\begin{corollary}\label{coro:agg-finite}
Suppose assumptions~\ref{assume:diff} and \ref{assume:parameters} hold. 
After a finite number of calls to Algorithm~\ref{alg:simple-agg-step}, starting from a point that satisfies \eqref{restate:eq:barrier-primal-sequence-nice}, Algorithm~\ref{simple-one-phase} will terminate.
\end{corollary}

\begin{proof}
For any successful step size $\alpha_{P}$ by \eqref{simple-min-step-size-aggresssive} for an aggressive step and $\frac{s_i y_i}{\mu} \in [\parComp, 1 / \parComp]$ we have
\begin{flalign*}
\alpha_{P} \ge \minStepFunc(\mu, s)  &= \min \left\{ 1/2, \frac{(\parCompAgg - \parComp)}{2  \parCompAgg \mu }  \min_{i \in \{ 1, ..., \ncon \} : \conWeight_i > 0}{\frac{s_i}{w_i}}  \right\} \\
&\ge \min\left\{ 1/2, \frac{(\parCompAgg - \parComp)}{2  \parCompAgg^2 }  \min_{i \in \{ 1, ..., \ncon \} : \conWeight_i > 0}{\frac{1}{y_i w_i}}  \right\}.
\end{flalign*}
Since Lemma~\ref{lem:yw-bounded} bounds $y^T \conWeight$ from above and $y, \conWeight \ge 0$ we deduce $\alpha_{P}$ is bounded away from zero.
We reduce $\mu$ by at least $\alpha_{P} \mu$ each call to Algorithm~\ref{alg:aggressive} that terminates with $\status = \success$. Furthermore, for sufficiently small $\mu$, whenever \eqref{agg-criteron} holds the optimality criterion \eqref{terminate-kkt} is satisfied. Combining these facts proves the Lemma.
\end{proof}

However, Corollary~\ref{coro:agg-finite} does not rule out the possibility that there is an infinite number of consecutive stabilization steps. Ruling out this possibility is the purpose of Lemma~\ref{lemConsecutiveStable}.

\subsubsection{Convergence of stabilization steps}\label{conv:stb}

This subsection is devoted to showing that consecutive stabilization steps eventually satisfy the criterion for an aggressive step or the unboundedness criterion is satisfied.

\begin{restatable}{lemma}{lemConsecutiveStable}\label{lemConsecutiveStable}
Suppose assumptions~\ref{assume:diff} and \ref{assume:parameters} hold. 
After a finite number of consecutive stabilization steps within Algorithm~\ref{simple-one-phase}, either the aggressive criterion~\eqref{agg-criteron} or the unboundedness termination criterion~\eqref{terminate-dual-infeasible} is met.
\end{restatable}

The proof of Lemma~\ref{lemConsecutiveStable} is given in Section~\ref{sec:lemConsecutiveStable}. Let us sketch the main ideas. First, we can show that the set of iterates that stabilization steps generate remain in a compact set. We can use this to uniformly bound quantities such as Lipschitz constants. The crux of the proof is showing there cannot be an infinite number of consecutive stabilizations steps because this will imply the augmented log barrier merit function tends to negative infinity and the unboundedness termination criterion holds.

\subsubsection{Main result}\label{subsec:main-result}

We now state Theorem~\ref{thm:global-convergence}, the main theoretical result of the paper.

\begin{theorem}\label{thm:global-convergence}
Suppose assumptions~\ref{assume:diff} and \ref{assume:parameters} hold. Algorithm~\ref{simple-one-phase} terminates after a finite number of iterations.
\end{theorem}

\begin{proof}
Corollary~\ref{coro:agg-finite} shows that the algorithm must terminate after a finite number of aggressive steps. Lemma~\ref{lemConsecutiveStable} shows that the algorithm terminates or an aggressive step must be taken after a finite number of stabilization steps. The result follows.
\end{proof}

\section{A practical one-phase algorithm}\label{sec:practical-alg}

Section~\ref{sec:simple-alg} presented a simple one-phase algorithm that is guaranteed to terminate (eventually) with a certificate of unboundedness, infeasibility, or optimality. However, to simplify, we omitted several practical details including the following.
\begin{enumerate}
\item The introduction of inner iterations that reuse the factorizations of $\Schur$. This reduces the total number of factorizations (Section~\ref{schur-reuse}).
\item How to choose the step sizes $\alpha_{P}$ and $\alpha_{D}$ in a practical manner (Section~\ref{step-acceptance}).
\item Using a filter to encourage steps that significantly reduce the KKT error (Section~\ref{sec:filter}).
\item How to compute $\delta$ in a practical manner (Section~\ref{sec:practical-alg-outline} and Appendix~\ref{sec:mat-fact}).
\item How choose the initial variable values (Appendix~\ref{sec:initialization}).
\end{enumerate}
The full algorithm is described in Section~\ref{sec:practical-alg-outline} with appropriately modified stabilization and aggressive steps. For complete details see the implementation at \url{https://github.com/ohinder/OnePhase}.

\subsection{Reusing the factorization of $\Schur$}\label{schur-reuse}

Our practical algorithm (Algorithm~\ref{practical-one-phase-IPM}) consists of inner iterations and outer iterations. The inner iterations reuse the factorization of the matrix $\Schur$ to compute directions as follows. Let $(\mu, x, s, y)$ be the current iterate and $(\hat{\mu}, \hat{x}, \hat{s}, \hat{y})$ be the iterate at the beginning of the outer iteration, i.e., where $\Schur$ was evaluated. Then the directions are computed during each inner iteration as follows:
\begin{subequations}\label{practical-direction}
\begin{flalign}
(\Schur + \delta I)  \dir{x} &= -\left( b_{D} + \grad \cons(\hat{x})^T \hat{S}^{-1} \left( Y b_{P} - b_{C} \right) \right) \label{practical-dx} \\
\dir{s} &\gets -(1 - \gamma) \mu \conWeight - \grad \cons(x)  \dir{x} \label{practical-compute-ds}  \\
\dir{y} &\gets  -\hat{S}^{-1} \hat{Y} (\grad \cons(\hat{x})  \dir{x} + b_{P} - \hat{Y}^{-1} b_{C})  \label{practical-compute-ds} .
\end{flalign}
\end{subequations}
This is identical to the direction computation described in Section~\ref{sub:direction-computation} if $(\mu, x, s, y) = (\hat{\mu}, \hat{x}, \hat{s}, \hat{y})$.

\subsection{Step size choices and acceptance}\label{step-acceptance}
First we specify a criterion to prevent the slack variables from getting too close to the boundary. In particular, given any candidate primal iterate $(x^{+}$, $s^{+})$ we require that the following \fracBound{} rule be satisfied:
\begin{flalign}\label{fracBoundary-primal}
s^{+} \ge  \parFracBoundary \min\{ s, \| d_{x} \|_{\infty} \left( \delta + \| d_{y} \|_{\infty} + \| d_{x} \|_{\infty}^{\parFracBoundaryExp} \right) \ones \},
\end{flalign}
where $\parFracBoundary$ is a diagonal matrix with entries $\parFracBoundary_{i,i} \in \parFracBoundaryInterval$ with default entry values of $\parFracBoundaryValue$  and $\parFracBoundaryExp \in \parFracBoundaryExpInterval$ with default value of $\parFracBoundaryExpValue$. Note that the $\| d_{x} \|_{\infty} \left( \delta + \| d_{y} \|_{\infty} + \| d_{x} \|_{\infty}^{\parFracBoundaryExp} \right)$ term plays a similar role to $\mu$ in the more typical $s^{+} \ge  \parFracBoundary s \min\{ 1 , \mu \}$ \fracBound{} rule (say of IPOPT \cite{wachter2006implementation}) in that it allows (almost) unit steps in the superlinear convergence regime. 

In both the aggressive steps and stabilization steps we use a backtracking line search. 
We choose the initial trial primal step size $\alpha_{P}^{\max}$ to be the maximum $\alpha_{P} \in [0,1]$ that satisfies the \fracBound{} rule:
\begin{flalign}\label{fracBoundaryPrimalMax}
s + \alpha_{P} \dir{s} &\ge  \parFracBoundaryMax  \min\{ s, \| d_{x} \|_{\infty} \left( \delta + \| d_{y} \|_{\infty} + \| d_{x} \|_{\infty}^{\parFracBoundaryExp} \right)  \ones \},
\end{flalign}
where the parameter $\parFracBoundaryMax$ is a diagonal matrix with entries $\parFracBoundaryMax_{i,i} \in \parFracBoundaryMaxInterval$. The default value of $\parFracBoundaryMax_{i,i}$ is $\parFracBoundaryMaxValueNL$ for nonlinear constraints and $\parFracBoundaryMaxValueLinear$ for linear constraints. The idea of this choice for $\alpha_{P}^{\max}$ is that the \fracBound{} rule \eqref{fracBoundary-primal} is likely to be satisfied for the first trial point, i.e., $\alpha_{P} = \alpha_{P}^{\max}$ when $\| d_{x} \|$ is small. To see this, note that by differentiability of $a$ we have $\| s + d_{s} - s^{+} \| \le \| a(x) + \grad a(x) \dir{x} - a(x + \dir{x}) \| = O(\| \dir{x} \|^2)$. Furthermore, the right hand side of \eqref{fracBoundary-primal} and \eqref{fracBoundaryPrimalMax} are identical except for $\parFracBoundary$ and $\parFracBoundaryMax$. Hence if $\dir{x} \rightarrow 0$, \eqref{fracBoundaryPrimalMax} holds and $\parFracBoundaryMax_{i,i} >  \parFracBoundary_{i,i}$ for $i$ corresponding to nonlinear constraints then \eqref{fracBoundary-primal} is satisfied in the limit for $\alpha_{P} = \alpha_{P}^{\max}$.

It remains to describe how to update the dual variables. Given some candidate primal iterate $(x^{+}$, $s^{+})$,  let $B( s^{+}, \dir{y} )$ be the set of feasible dual step sizes. More precisely,  we define $B( s^{+}, \dir{y} ) \subseteq [0,1]$ to be the largest interval such that if $\alpha_{D} \in B( s^{+}, \dir{y} )$ then
\begin{subequations}\label{subeq:set-B}
\begin{flalign}
 \frac{s^{+}_i (y + \alpha_{D} \dir{y})_i}{\mu^{+}} &\in [\parComp, 1/\parComp ] ~~ \forall i \in \{ 1, \dots, m \} \label{satisfy-comp} \\
y + \alpha_{D} \dir{y} &\ge  \parFracBoundary y \min\{ 1 , \| \dir{x} \|_{\infty} \}. \label{fracBoundary-dual} %
\end{flalign}
\end{subequations}
If there is no value of $\alpha_{D}$ satisfying \eqref{subeq:set-B} we set $B( s^{+}, \dir{y} )$ to the empty set and the step will be rejected. Recall the parameter $\parComp \in (0,1)$ was defined in \eqref{eq:comp-slack}. The purpose criteria \eqref{satisfy-comp} is to ensure the algorithm always satisfy \eqref{eq:comp-slack}, i.e., $\frac{S y}{\mu} \in [\ones \parComp, e/\parComp ]$. Equation~\eqref{fracBoundary-dual} is a \fracBound{} condition for the dual variables.
 We compute the dual step size as follows:
\begin{subequations}\label{eq:compute-alpha-D}
\begin{flalign}
\alpha_{D} &\gets \arg \min_{\zeta \in B( s^{+}, \dir{y} )} \| S^{+} y - \mu^{+} + \zeta S^{+} \dir{y} \|^2_{2} + \| \grad \obj(x^{+})   +  \grad \cons(x^{+})^T (y + \zeta \dir{y}) \|^{2}_{2} %
\label{eq:alphaD-least-squares} \\
 \alpha_{D} &\gets \min\{ \max\{ \alpha_{D}, \alpha_{P} \}, \max B( s^{+}, \dir{y} ) \}. \label{min:alpha-D}
\end{flalign}
\end{subequations}
Equation~\eqref{eq:alphaD-least-squares} can be interpreted as choosing the step size $\alpha_{D}$ that minimizes the complementarity and dual infeasibility. This reduces to a one-dimensional least squares problem in $\zeta$ which has a closed form expression for the solution. Equation~\eqref{min:alpha-D} encourages the dual step size to be at least as large as the primal step size $\alpha_{P}$. This prevents tiny dual step sizes being taken when the dual direction is not a descent direction for the dual infeasibility, which may occur if $\delta$ is large. %

\subsection{A filter using a KKT merit function}\label{sec:filter}

In the stabilization search directions we accept steps that make progress on one of two merit functions, which form a filter. The first function $\phi_{\mu}$ is defined in Section~\ref{sec:simple-stable}. The second function, we call it the KKT merit function, measures the scaled dual feasibility and complementarity:
\begin{flalign}\label{merit-KKT}
\meritKKT_{\mu} ( x, s, y )  = \sigma( y ) \max\{ \| \grad \Lag_{\mu}(x, y ) \|_{\infty} ,  \| S y - \mu \ones \|_{\infty} \}.
\end{flalign}
This merit function measures progress effectively in regimes where $\Schur$, given in \eqref{eq:Schur-matrix}, is positive definite. In this case, the search directions generated by \eqref{eq:Schur-complement-system} will be a descent direction on this merit function (for the first inner iteration of each outer iteration of Algorithm~\ref{practical-one-phase-IPM} i.e., $j = 1$). This merit function is similar to the potential functions used in interior point methods for convex optimization \cite{andersen1998computational,huang2016solution}. Unfortunately, while this merit function may be an excellent choice for convex problems, in nonconvex optimization it has serious issues. In particular, the search direction \eqref{eq:Schur-complement-system} might not be a descent direction. Moreover, changing the search direction to minimize the dual feasibility has negative ramifications. The algorithm could converge to a critical point of the dual feasibility where $\meritKKT_{\mu} ( x, s, y ) \neq 0$\footnote{To see why this occurs one need only consider an unconstrained problem, e.g., minimizing $\obj(x) = x^4 + x^3 + x$ subject to no constraints. The point $x = 0$ is a stationary point for the gradient of $\grad \obj(x)$, but is not a critical point of the function.}. For further discussion of these issues, see \cite{shanno2000interior}.

While it is sufficient to guarantee convergence by accepting steps if \eqref{eq:phi-sufficient-progress} is satisfied, in some regimes e.g., when $\Schur$ is positive definite, this may select step sizes $\alpha_{P}$ that are too conservative;  for example, near points satisfying the sufficient conditions for local optimality. In these situations the KKT error is often a better measure of progress toward a local optimum than a merit function that discards information about the dual feasibility. Furthermore, from our experience, during convergence toward an optimal solution, numerical errors in the function $\phi_{\mu}$ may cause the algorithm to fail to make sufficient progress on the merit function $\phi_{\mu}$, i.e., \eqref{eq:phi-sufficient-progress} is not satisfied for any $\alpha_{P}$. For these reasons we decide to use a filter approach \cite{fletcher2002nonlinear,wachter2006implementation}. Typical filter methods \cite{fletcher2002nonlinear} require progress on either the constraint violation or the objective function. Our approach is distinctly different, because we accept steps that make progress on either the merit function $\phi_{\mu}$ or the merit function $\meritKKT_{\mu}$.
To be precise, we accept any iterate $(\mu^{+}, x^{+}, s^{+}, y^{+})$ that makes sufficient progress on the augmented log barrier function $\phi_{\mu}$, or satisfies the equations
\begin{subequations}\label{eq:filter}
\begin{flalign}
\meritKKT_{\mu} (x^{+}, s^{+}, y^{+}) &\le (1 - \parKKTReductFactor \alpha_{P} ) \meritKKT_{\mu} (\tilde{x}, \tilde{s}, \tilde{y}) \label{eq:kkt-progress} \\
\phi_{\mu}(x^{+}, s^{+}, y^{+}) &\le \phi_{\mu}(\tilde{x}, \tilde{s}, \tilde{y}) + \sqrt{\meritKKT_{\mu} (\tilde{x}, \tilde{s}, \tilde{y})} %
\end{flalign} %
\end{subequations}
 for every previous iterate $(\tilde{\mu}, \tilde{x}, \tilde{s}, \tilde{y})$ with $\cons(\tilde{x}) + \tilde{s} = \cons(x) + s$.

The idea of \eqref{eq:filter} is that for points with similar values of the augmented log barrier function the KKT error is a good measure of progress. However, we want to discourage the algorithm from significantly increasing the augmented log barrier function while reducing the KKT error because if this is occurring, then the algorithm might converge to a saddle point.

\subsection{Algorithm outline}\label{sec:practical-alg-outline}

The general idea of Algorithm~\ref{practical-one-phase-IPM} follows. At each outer iteration we factorize the matrix $\Schur + \delta \eye$ with an appropriate choice of $\delta$ using Algorithm~\ref{alg:mat-fact}. With this factorization fixed, we attempt to take multiple inner iterations (at most $\parNumCor$), which corresponds to solving system~\eqref{primal-dual-Newton-direction} with different right-hand side choices but the same matrix $\Schur + \delta \eye$. Each inner iteration is either an aggressive step or a stabilization step. If, on the first inner iteration, the step fails (i.e., due to a too small step size), we increase $\delta$ and refactorize $\Schur + \delta \eye$. Note that we evaluate the Hessian of the Lagrangian once per outer iteration. The selection of the initial point $(\mu^{0}, x^{0}, s^{0}, y^{0})$ is described in Section~\ref{sec:initialization}.

\begin{algorithm}[H]
\textbf{Input:} some initial point $x^{0}$, vector $\conWeight \ge 0$, and variables $\mu^0, s^{0}, y^{0} > 0$  such that $\cons(x^{0}) + s^{0} = \conWeight \mu^{0}$ and equation~\eqref{eq:comp-slack}   is satisfied with $k=0$. Termination tolerances $\TOLopt \in \TOLoptInterval$, $\TOLinfOne \in \TOLinfOneInterval$, $\TOLinfTwo \in \TOLinfTwoInterval$ and $\TOLunbounded \in \TOLunboundedInterval$. \\
\textbf{Output:} a point $(\mu, x, s, y)$ that satisfies at least one of the inequalities \termination{}
\vspace{0.1 cm} %
\begin{enumerate}[label*=A.{\arabic*}]
\item \label{line:init-delta} \emph{Initialize.} Set $\delta \gets 0$.
\item \label{line:form-K}  \emph{New outer iteration.} \\
Set $(\hat{\mu}, \hat{x}, \hat{s}, \hat{y}) \gets (\mu, x, s, y)$. Form the matrix $\Schur$ using \eqref{eq:Schur-matrix}. Set $\deltaPrev \gets \delta$.
\item \label{line:factor-schur} \emph{Select $\delta$ and factorize the matrix $\Schur + \delta \eye$,} \\
i.e., run Algorithm~\ref{alg:mat-fact} (see Appendix~\ref{sec:mat-fact}) with: \\
\hspace*{0.1cm}  \textbf{Input:} $\Schur$, $\delta$. \\
\hspace*{0.1cm} \textbf{Output:} New value for $\delta$, factorization of $\Schur + \delta \eye$.
\item \label{take-steps}  \emph{Perform inner iterations where we recycle the factorization of $\Schur + \delta \eye$.} \\
For $j \in \{ 1, \dots , \parNumCor \}$ do:
\begin{enumerate}[label*=.{\arabic*}]
\item \emph{Check termination criterion}. \\
If any of the inequalities \termination{} holds at the point $(\mu,x,s,y)$, terminate.
\item \emph{Take step}\label{line:take-step}
\begin{enumerate}[label=-Case {\Roman*}]
\item If the aggressive step criterion~\eqref{agg-criteron} is satisfied, do an aggressive step, \\
i.e., run Algorithm~\ref{alg:aggressive} with: \\
\hspace*{0.1cm}  \textbf{Input:} the matrix $\Schur + \delta \eye$, its factorization, the point $(\mu, x, s, y)$ and $(\hat{\mu}, \hat{x}, \hat{s}, \hat{y})$. \\
\hspace*{0.1cm}  \textbf{Output:} A $\status$ and a new point $(\mu^{+},x^{+},s^{+},y^{+})$.
\item Otherwise, do a stabilization step, \\
i.e., run Algorithm~\ref{alg:stable} with: \\
\hspace*{0.1cm} \textbf{Input:} the matrix $\Schur + \delta \eye$, its factorization, the point $(\mu, x, s, y)$ and $(\hat{\mu}, \hat{x}, \hat{s}, \hat{y})$. \\
\hspace*{0.1cm} \textbf{Output:} A $\status$ and a new point $(\mu^{+},x^{+},s^{+},y^{+})$.
\end{enumerate}
\item \emph{Deal with failures}. \\
If $\status = \success$ set $(\mu, x, s, y) \gets ( \mu^{+}, x^{+},s^{+},y^{+})$. If $\status = \failure$ and $j = 1$ go to \eqref{increase-delta-for-failure}.  If $\status = \failure$ and $j > 1$ go to step~\eqref{line:form-K}.
\end{enumerate}
\item 
Go to \eqref{line:form-K}.
\item \label{increase-delta-for-failure} \emph{Increase $\delta$ to address failure.} \\
Set $\delta = \max\left\{\parDeltaIncreaseFailure \delta, \quad \parDeltaMin, \quad \deltaPrev \parDeltaDecrease, \quad \frac{\| \grad \Lag_{\mu}(x,y) \|_{\infty}}{ \| d_{x} \|_{\infty} } \right\}$. \\
If $\delta \le \parDeltaMax$ then factorize the matrix $\Schur + \delta \eye$ and go to step \eqref{take-steps}, otherwise terminate with $\status = \maxDelta$.
\end{enumerate}
\caption{A practical one-phase IPM}\label{practical-one-phase-IPM}
\end{algorithm}

\subsubsection{Aggressive steps}

Recall that when computing aggressive search directions we solve system~\eqref{primal-dual-Newton-direction} with $\gamma = 0$; that is, we aim for feasibility and optimality simultaneously. We accept any step size assuming it satisfies the \fracBound{} rule \eqref{fracBoundary-primal} and the set of valid dual step sizes is non-empty: $B( s^{+}, \dir{y} ) \neq \emptyset$ (see equations~\eqref{subeq:set-B}). 

To prevent unnecessary line searches, we only attempt an aggressive line search if

\begin{flalign}
\grad \Lag_{\gamma \mu}(x,\tilde{y})^T \dir{x} < 0 \label{eq:agg-could-improve}
\end{flalign}
where $\tilde{y} = S^{-1} \mu ( \ones \gamma  - (1 - \gamma ) Y \conWeight )$. Note that \eqref{eq:agg-could-improve} always holds if $(\hat{x}, \hat{s},  \hat{y}, \hat{\mu}) = (\mu, x, s, y)$.

The backtracking line search of the aggressive step has a minimum step size. If during the backtracking line search (line~\ref{line:agg-back-too-small} of Algorithm~\ref{alg:aggressive}) the step size $\alpha_{P}$ is smaller than
 \begin{flalign}\label{min-step-size-aggresssive}
\bar{\minStepFunc}(\mu, s) := \min \left\{ 1/2, \frac{\parBacktracking}{4 \mu} \times \min\left\{ \frac{\parCompAgg - \parComp}{\parCompAgg}, 1 - \parFracBoundary_{i,i} \right\} \times  \min_{i \in \{ 1, ..., \ncon \} : \conWeight_i > 0}{\frac{s_{i}}{\conWeight_i} } \right\}
\end{flalign}
then we immediately reject the step and exit Algorithm~\ref{alg:aggressive}. Note that the function $\bar{\minStepFunc}(\mu, s)$ is a more sophisticated version of $\minStepFunc$ used for the simple algorithm and defined in \eqref{simple-min-step-size-aggresssive}. Following this, $\delta$ is increased in Line~\ref{increase-delta-for-failure} of Algorithm~\ref{practical-one-phase-IPM} and a new aggressive step is attempted. It is possible that $\delta$ will be increased many times; however, for sufficiently large $\delta$ an acceptable step will be found (see Lemma~\ref{lemma:agg-succeeds}). 

To choose $\gamma$ in line~\ref{mehrotra-heuristic} of Algorithm~\ref{alg:aggressive} we use a heuristic inspired by Mehrotra's predictor-corrector method. This heuristic only requires a direction computation --- we do not evaluate any functions of the nonlinear program.

It is possible that when we take a step that we reduce $\mu$ significantly but $\| \grad \Lag(x,y) \|_{\infty}$ remains large. Line~\ref{agg-protect-1} and \ref{agg-protect-1} of Algorithm~\ref{alg:aggressive} guard against this possibility. This scheme is controlled by the parameter $\parAggProtect \in \parAggProtectInterval$ with a default value of $\parAggProtectValue$. Values of $\parAggProtect$ close to $1$ give no guarding and close to $0$ are conservative.

\begin{algorithm}[H]
\textbf{Input:} The matrix $\Schur + \delta \eye$, its factorization, the current point $(\mu, x, s, y)$ and the point $(\hat{x}, \hat{s},  \hat{y}, \hat{\mu})$ from the beginning of the outer iteration.  \\
\textbf{Output:} A new point $(\mu^{+}, x^{+}, s^{+}, y^{+})$ and a $\status$ of either $\success$ or $\failure$.
\begin{enumerate}[label*=A.{\arabic*}]
\item \label{mehrotra-heuristic} \emph{Use Mehrotra's predictor-corrector heuristic to choose $\gamma$.} 
\begin{enumerate}[label*=.{\arabic*}]
\item Compute the vector $b$ at the point $(\mu, x, s, y)$ via \eqref{def:b} with $\gamma = 0$.
\item Compute the search direction $(\dir{x}, \dir{s}, \dir{y})$ using \eqref{practical-direction}.
\item Estimate the largest primal step size $\alpha^{\max}_{P}$ from equation~\eqref{fracBoundaryPrimalMax}.
\item Set $\gamma \gets \min\{0.5, (1 - \alpha_{\max})^2 \}$.
\end{enumerate}
\item Compute the vector $b$ at the point $(\mu, x, s, y)$ via \eqref{def:b} with $\gamma$ as chosen in previous step.
\item Compute the search direction $(\dir{x}, \dir{s}, \dir{y})$ using \eqref{practical-direction}.
\item \emph{Check that the direction has a reasonable chance of being accepted.} If \eqref{eq:agg-could-improve} is not satisfied then terminate with $\status = \failure$.
\item Estimate the largest primal step size $\alpha^{\max}_{P}$ from equation~\eqref{fracBoundaryPrimalMax}.
\item \label{agg:line:back-track} \emph{Perform a backtracking linesearch}. Starting from $\alpha_{P} \gets \alpha^{\max}_{P}$ repeat the following steps.
\begin{enumerate}[label*=.{\arabic*}]
\item \label{prac-agg:line:back-too-small} \emph{Check the step size is not too small.} If $\alpha_{P} \le \bar{\minStepFunc}(\mu, s)$ then goto line~\ref{agg-prac:line:terminate}.
\item Compute the trial primal variables $(\mu^{+}, x^{+}, s^{+})$ via~\eqref{eq:iterate-update}.
\item If the \fracBound{} rule \eqref{fracBoundary-primal} is not satisfied, then set $\alpha_{P} \gets \parBacktracking \alpha_{P}$ and go to line~\ref{prac-agg:line:back-too-small}.
\item Compute feasible dual step sizes $B( s^{+}, \dir{y} )$.
\item If $B( s^{+}, \dir{y} ) = \emptyset$ then trial step has failed, then set $\alpha_{P} \gets \parBacktracking \alpha_{P}$ and go to line~\ref{prac-agg:line:back-too-small}.
\item Compute dual variable step size $\alpha_{D}$ using \eqref{eq:compute-alpha-D} and compute the trial dual variables $y^{+}$ using \eqref{eq:update-y}.
\item \label{agg-protect-1} If $\mu^{+} / \mu \ge 1 - \parAggProtect$ goto line~\ref{prac-agg:line:success}.
\item \label{agg-protect-2} Let $\tau \gets \frac{\mu^{+}}{(1 - \parAggProtect) \sigma(y^{+}) \| \grad \Lag_{\mu^{+}}(x^{+},y^{+}) \|_{\infty}}$. If $\tau < 1$ then set $\alpha_{P} \gets \max\{ \parAggProtect^2 , \alpha_{P} \tau^2 \}$ and go to line~\ref{prac-agg:line:back-too-small}.
\item\label{prac-agg:line:success} Terminate with $\status = \success$ and return the point $(\mu^{+}, x^{+}, s^{+}, y^{+})$.
\end{enumerate}
\item \label{agg-prac:line:terminate} Terminate with $\status = \failure$.
\end{enumerate}
\caption{Practical aggressive step}\label{alg:aggressive}
\end{algorithm}

\subsubsection{Stabilization steps}

During the backtracking line search we terminate with $\status = \failure$ if
\begin{flalign}\label{eq:min-step-size-stable}
\alpha_{P} \le \parMinStableStepSize
\end{flalign}
where $\parMinStableStepSize \in \parMinStableStepSizeInterval$ with default value $\parMinStableStepSizeValue$.

We then exit Algorithm~\ref{alg:stable} and go to line~\ref{increase-delta-for-failure} of Algorithm~\ref{practical-one-phase-IPM}, where we increase $\delta$ and attempt a new stabilization step. From Lemma~\ref{lemConsecutiveStable} we know for sufficiently large $\delta$ the stabilization step will succeed.

To prevent unnecessary line searches, we only attempt a stabilization line search if

\begin{flalign}
\grad \barrier_{\mu}(x)^T \dir{x} < 0. \label{eq:obj-could-improve}
\end{flalign}
The idea is to take steps only when it is possible to decrease $\barrier_{\mu}$. This condition is always satisfied if Algorithm~\ref{alg:stable} is called from the first inner iteration of Algorithm~\ref{practical-one-phase-IPM}. It may not be satisfied when the inner iteration is greater than one because we recycle the factorization of $\Schur + \delta \eye$.

\begin{algorithm}[H]
\textbf{Input:} The matrix $\Schur + \delta \eye$, its factorization, the current point $(\mu, x, s, y)$ and the point $(\hat{x}, \hat{s},  \hat{y}, \hat{\mu})$ from the beginning of the outer iteration.   \\
\textbf{Output:} A new point $(\mu^{+}, x^{+}, s^{+}, y^{+})$ and a $\status$ of either $\success$ or $\failure$
\begin{enumerate}[label*=A.{\arabic*}]
\item Compute the vector $b$ at the point $(\mu, x, s, y)$ via \eqref{def:b} with $\gamma = 1$.
\item Compute the search direction $(\dir{x}, \dir{s}, \dir{y})$ by solving \eqref{eq:Schur-complement-system}, \eqref{compute-ds} and \eqref{compute-dy} respectively.
\item \emph{Check that the direction has a reasonable chance of being accepted.} If \eqref{eq:obj-could-improve} is not satisfied then terminate with $\status = \failure$.
\item Estimate the largest primal step size $\alpha^{\max}_{P}$ from equation~\eqref{fracBoundaryPrimalMax}.
\item \label{agg:line:back-track} \emph{Perform a backtracking linesearch}. Starting from $\alpha_{P} \gets \alpha^{\max}_{P}$ repeat the following steps.
\begin{enumerate}[label*=.{\arabic*}]
\item \label{line:agg-back-too-small} \emph{Check the step size is not too small.} If $\alpha_{P} \le \parMinStableStepSize$ then goto line~\ref{line:agg-terminate}.
\item Compute the trial primal variables $(\mu^{+}, x^{+}, s^{+})$ via~\eqref{eq:iterate-update}.
\item If the \fracBound{} rule \eqref{fracBoundary-primal} is not satisfied, then set $\alpha_{P} \gets \parBacktracking \alpha_{P}$ and go to line~\ref{line:agg-back-too-small}. 
\item Compute feasible dual step sizes $B( s^{+}, \dir{y} )$.
\item If $B( s^{+}, \dir{y} ) = \emptyset$, then set $\alpha_{P} \gets \parBacktracking \alpha_{P}$ and go to line~\ref{line:agg-back-too-small}.
\item Compute dual variable step size $\alpha_{D}$ using \eqref{eq:compute-alpha-D} and compute the trial dual variables $y^{+}$ using \eqref{eq:update-y}.
\item \label{line:filter} \emph{Sufficient progress on filter.} If neither equation~\eqref{eq:phi-sufficient-progress} or \eqref{eq:filter} is satisfied, then set $\alpha_{P} \gets \parBacktracking \alpha_{P}$ and go to line~\ref{line:agg-back-too-small}.
\item Terminate with $\status = \success$ and return the point $(\mu^{+}, x^{+}, s^{+}, y^{+})$.
\end{enumerate}
\item \label{line:agg-terminate} Terminate with $\status = \failure$.
\end{enumerate}
\caption{Practical stabilization step}\label{alg:stable}
\end{algorithm}

\subsection{Algorithm Parameters}

\begin{table}[H]
\begin{tabular}{ |c| p{7.5cm}|p{2.5cm}| p{2.7cm}| } 
 \hline
Parameter & Description & Possible values & Chosen value  \\ 
\hline
$\TOLopt$ & Tolerance for the optimality criterion \eqref{terminate-kkt}. & $\TOLoptInterval$ & $\TOLoptValue$  \\
\hline
$\TOLinfOne$ & Tolerance for the infeasibility criterion \eqref{terminate-primal-infeasible} corresponding to $\infeasFuncOne$. & $\TOLinfOneInterval$ & $\TOLinfOneValue$ \\
\hline
$\TOLinfTwo$ & Tolerance for the infeasibility criterion \eqref{terminate-primal-infeasible} corresponding to $\infeasFuncTwo$. & $\TOLinfTwoInterval$ & $\TOLinfTwoValue$ \\
\hline
$\TOLunbounded$ & Tolerance for the unboundedness criterion \eqref{terminate-dual-infeasible}. & $\TOLunboundedInterval$ & $\TOLunboundedValue$ \\
\hline
$\parConRegularizer$ & Used to modify the log barrier function to prevent primal iterates from diverging. & $\parConRegularizerInterval$ & $\parConRegularizerValue$ \\
\hline
  $\parComp$ & Restricts how far complementarity of $s$ and $y$ can be from $\mu$. See \eqref{eq:comp-slack}.  &$\parCompInterval$ & $\parCompValue$ \\ 
 \hline
   $\parCompAgg$ & Restricts how far complementarity of $s$ and $y$ can be from $\mu$ in order for the aggressive criterion to be met. See \eqref{agg-criteron-buffer}.  & $\parCompAggInterval$ & $\parCompAggValue$  \\ 
\hline
\hline
            $\parObjReductFactor$ & Acceptable reduction factor for the merit function $\phi_{\mu}$ during stabilization steps. See \eqref{eq:phi-sufficient-progress}.  & $\parObjReductFactorInterval$ & $\parObjReductFactorValue$  \\
    \hline
   $\parMinStableStepSize$ & Minimum step size for stable line searches. See \eqref{eq:min-step-size-stable}.  & $\parMinStableStepSizeInterval$ & $\parMinStableStepSizeValue$  \\ 
   \hline 
      $\parKKTReductFactor$ & Acceptable reduction factor for the scaled KKT error $\meritKKT_{\mu}$ during stabilization steps. See \eqref{eq:kkt-progress}.  & $\parKKTReductFactorInterval$ & $\parKKTReductFactorValue$ \\ 
      \hline
    $\parBacktracking$ & Backtracking factor for line searches in Algorithm~\ref{alg:aggressive} and \ref{alg:stable}. & $\parBacktrackingInterval$ & $\parBacktrackingValue$ \\
    \hline
$\parFracBoundaryExp$ & Exponent of $\| \dir{x} \|$ used in \fracBound{} formula \eqref{fracBoundaryPrimalMax} for computing the maximum step size $\alpha_{P}^{\max}$. &$\parFracBoundaryExpInterval$ & $\parFracBoundaryExpValue$ \\
\hline
 $\parAggProtect$ & Step size for which we check that the dual feasibility is decreasing (line~\ref{agg-protect-1} and \ref{agg-protect-2} of Algorithm~\ref{alg:aggressive}). & $\parAggProtectInterval$ & \parAggProtectValue \\
\hline
 $\parFracBoundary$ & Diagonal matrix with the \fracBound{} parameter for each constraint. See \eqref{fracBoundary-primal} and \eqref{fracBoundary-dual}. & Each element is in the interval $\parFracBoundaryInterval$. & $\parFracBoundaryValue$ for all elements \\ 
        \hline
$\parFracBoundaryMax$ & Diagonal matrix with the \fracBound{} parameter for each constraint used in \eqref{fracBoundaryPrimalMax} for computing $\alpha_{P}^{\max}$. & Each element is in the interval $\parFracBoundaryMaxInterval$. & $\parFracBoundaryMaxValueLinear$ and $\parFracBoundaryMaxValueNL$ for linear and nonlinear constraints respectively \\
\hline
$\parDeltaMin$ & Used in Algorithm~\ref{practical-one-phase-IPM} and \ref{alg:mat-fact}.  & $\parDeltaMinInterval$ & $\parDeltaMinValue$ \\
\hline
$\parDeltaDecrease$ & Used in Algorithm~\ref{practical-one-phase-IPM} and \ref{alg:mat-fact}. & $\parDeltaDecreaseInterval$ & $\parDeltaDecreaseValue$ \\
\hline
$\parDeltaIncreaseFailure$ & Used in Algorithm~\ref{practical-one-phase-IPM}  and \ref{alg:mat-fact}. & $\parDeltaIncreaseFailureInterval$ & $\parDeltaIncreaseFailureValue$ \\
\hline
$\parDeltaMax$ &Used in Algorithm~\ref{practical-one-phase-IPM} and \ref{alg:mat-fact}. & $\parDeltaMaxInterval$ & $\parDeltaMaxValue$ \\
\hline
$\parNumCor$ & Maximum number inner iterations per outer iteration. See \eqref{take-steps} of Algorithm~\ref{practical-one-phase-IPM}.  & $\mathbb{N}$ & $\parNumCorValue$  \\ 
 \hline
 $\parInitialize$  & Minimum slack variable value in the initialization (Section~\ref{sec:initialization}). & $\parInitializeInterval$ & $\parInitializeValue$ \\
\hline
$\parInitializeMin$  & Minimum dual variable value in the initialization. & $\parInitializeMinInterval$ & $\parInitializeMinValue$ \\
\hline
$\parInitializeMax$  & Maximum dual variable value in the initialization. & $\parInitializeMaxInterval$ & $\parInitializeMaxValue$ \\
\hline
 $\parMuScale$ & Scales the size of $\mu^0$ in the initialization.  & $\parMuScaleInterval$ & $\parMuScaleValue$ \\
 \hline
\end{tabular}
\caption{Parameters values and descriptions}
\end{table}

\newpage

\section{Numerical results}\label{sec:numerical-results}

 The numerical results are structured as follows. Section~\ref{alg:comparison-IPOPT} compares our algorithm against IPOPT on CUTEst. Section~\ref{sec:infeas} compares our algorithm and IPOPT on a set of infeasible problems. For a comparison of the practical behavior of the dual multipliers we refer the reader to \cite{haeser2017behavior}. The code for our implementation can be found at \url{https://github.com/ohinder/OnePhase} and tables of results at \url{https://github.com/ohinder/OnePhase.jl/tree/master/benchmark-tables}.
For description of individual CUTEst problems see \url{http://www.cuter.rl.ac.uk/Problems/mastsif.shtml}.

Overall, on the problems we tested, we found that our algorithm required significantly less iterations to detect infeasibility and failed less often. However, given both solvers found an optimal solution then IPOPT tended to require fewer iterations.

\subsection{Comparison with IPOPT on CUTEst}\label{alg:comparison-IPOPT}

To obtain numerical results we use the CUTEst nonlinear programming test set \cite{gould2015cutest}. We selected a subset from CUTEst with more than $100$ variables and $100$ constraints, but the total number of variables and constraints less than $10,000$. We further restricted the CUTEst problems to ones that are classified as having first and second derivatives defined everywhere (and available analytically). This gave us a test set with $238$ problems. For solving our linear systems we use Julia's default Cholesky factorization (SuiteSparse) and IPOPT's default---the MUMPs linear solver.

Table~\ref{avg:evaluations} compares the number of calls and runtime of different elements of our algorithm. We do not currently compare the number of function evaluations with IPOPT, but in general we would anticipate that IPOPT needs slightly fewer evaluations per outer iteration. In particular, our algorithm may make multiple objective, constraint and Jacobian evaluations per inner iteration. We evaluate the Lagrangian of the Hessian once per outer iteration. For problems where function evaluations are expensive relative to factorization, one is often better off using SQP methods rather than interior point methods. For example, it is known that SNOPT generally requires fewer function evaluations than IPOPT \cite[Figure 2, Figure 3]{gill2015performance}. 

\begin{table}[H]
\begin{tabular}{l  p{3.0cm} l}
&   Mean \# calls per outer iteration & \% runtime contribution \\ 
Hessian & 1 & 3.4\% \\  
Schur complement &1 & 42.1\% \\
Jacobian & 2.2 & 7.9\%  \\
Gradient & 2.2 & 0.4\%  \\
Constraints & 6.9 & 0.6\% \\
Factorizations & 1.9  & 35.4 \%  \\
Backsolves & 10.0 & 0.8\% 
\end{tabular}
\caption{Number of calls and runtime of different algorithm elements.}\label{avg:evaluations}
\end{table}

We consider algorithm to have failed on a problem if it did not return a status of infeasible, optimal or unbounded. Overall the number of failures is 39 for IPOPT and 21 for the one-phase algorithm. From Table~\ref{compare-outputs} we can see that the one-phase algorithm detects infeasibility and KKT points more often than IPOPT. If we reduce the termination tolerances of both algorithms to $10^{-2}$ and re-run them on the test set the one-phase algorithm fails $10$ times versus $41$ times for IPOPT (IPOPT's algorithm makes choices based on the termination criterion).

\begin{table}[H]
\caption{Comparison on problems where solver outputs are different.}\label{compare-outputs}
\begin{tabular}{| c | c | c |}
& IPOPT & one-phase \\
\hline
failure &  32 & 14 \\ 
infeasible & 3 & 11 \\  
unbounded & 0 & 0 \\
KKT & 18 & 28 
\end{tabular}
\end{table}

To run this test set it takes a total of 16 hours for IPOPT and 10 hours for our one-phase algorithm\footnote{Computations were performed on one core and 8GB of RAM Intel(R) Xeon(R) CPU E5-2650 v2 at 2.60GHz using Julia 0.5 with one thread.}. However, IPOPT is generally significantly faster than our algorithm. It has a median runtime of $0.6$ seconds per problem versus $3.3$ seconds for our algorithm (including problems where the algorithm fails). This is not surprising since our code is written in Julia and is not optimized for speed; IPOPT is written in Fortran and has been in development for over 15 years. This speed difference is particularly acute on small problems where the overheads of Julia are large. Assuming factorization or computation of the hessian is the dominant computational cost we would expect the time per iteration to be similar if our algorithm was efficiently implemented in a compiled programming language. For this reason, we compare the algorithms based on iteration counts.
 
To compare iteration counts our graphs use the performance profiling of \citet*{dolan2002benchmarking}. In particular, on the $x$-axis we plot:
$$
\frac{\text{iteration count of solver}}{\text{iteration count of fastest solver}}
$$
and the curve plotted is a cumulative distribution over the test set. We can see from Figure~\ref{fig:iter-count-CUTEst} that the overall distribution of iteration counts seems similar to IPOPT. Although given both algorithms declare the problem optimal IPOPT is more likely to require fewer iterations (Figure~\ref{fig:iter-count-CUTEst-opt}).

\begin{figure}[H]
\includegraphics[scale=0.5]{\figDir{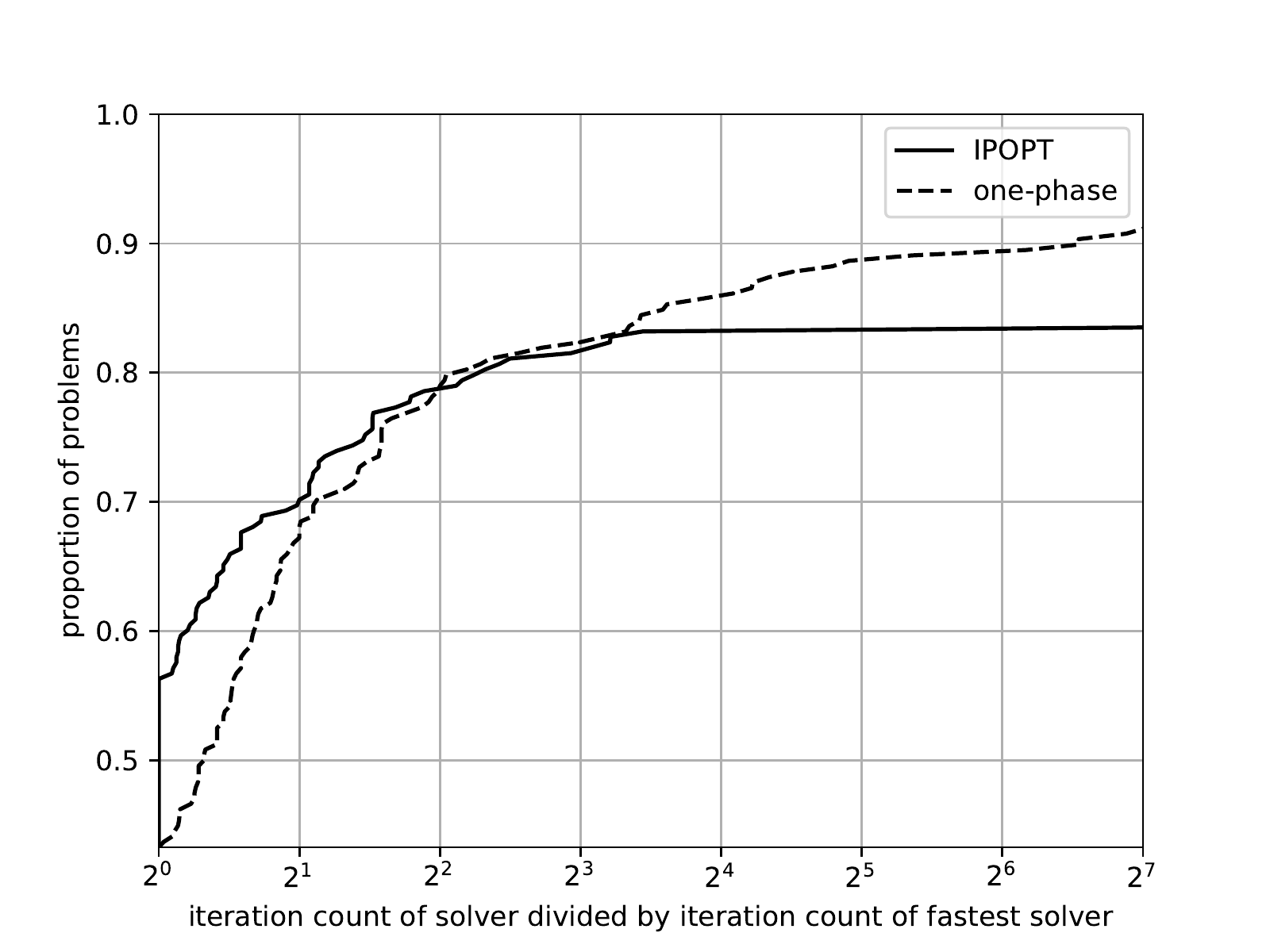}}
\caption{Performance profile on problems until algorithm succeeds (i.e., returns a status of optimal, primal infeasible, or unbounded)}\label{fig:iter-count-CUTEst}
\end{figure}

\begin{figure}[H]
\includegraphics[scale=0.5]{\figDir{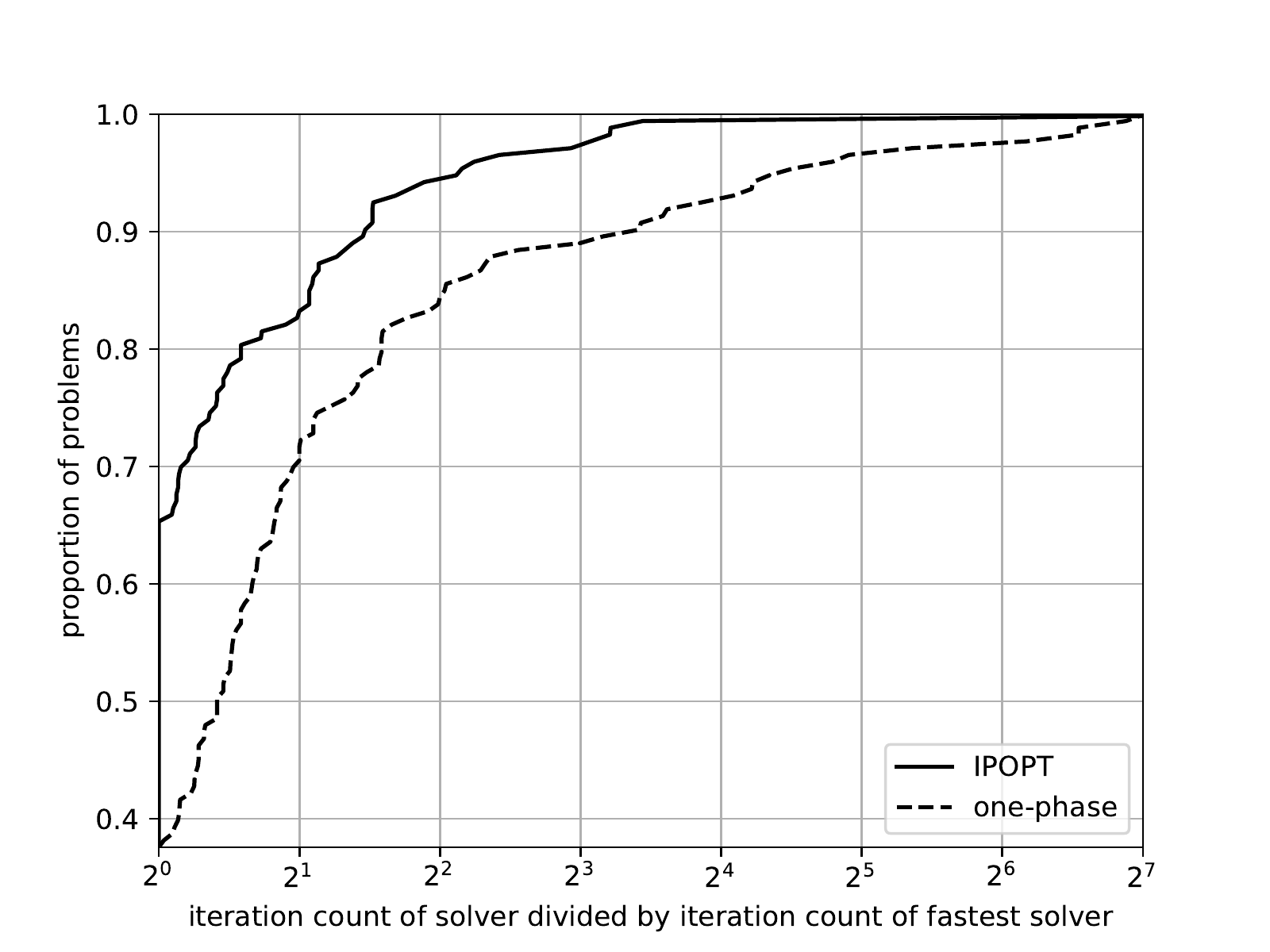}}
\caption{Performance profile on problems where both algorithms declare problem optimal.}\label{fig:iter-count-CUTEst-opt}
\end{figure}

\subsection{Comparison with IPOPT on infeasible problems}\label{sec:infeas}

On few of the CUTEst problems both solvers declare the problems infeasible. Therefore understand how the solvers perform on infeasible problems we perturb the CUTEst test set. Recall that CUTEst writes problems in the form $l \le c(x) \le u$. To generate a test set that was more likely to contain infeasible problems we shift the constraints (excluding the variable bounds) as follows:
$$
\tilde{c}(x) \gets c(x) + e,
$$
and then input the problems to the one-phase solver and IPOPT. Note that the new problems may or may not be feasible. In Table~\ref{table-perturbed-solver-output} we compare the solver outputs on this perturbed test set. 

\begin{table}[H]
\caption{Comparison on perturbed CUTEst where solver outputs are different.}\label{table-perturbed-solver-output}
\begin{tabular}{| c | c | c |}
& IPOPT & one-phase \\
\hline
failure &  42 & 18 \\ 
infeasible & 14 & 26 \\  
unbounded & 0 & 0 \\
KKT & 5 & 17 
\end{tabular}
\end{table}

Overall we found that $94$ of the $238$ problems were declared infeasible by both solvers. On this test set we compared the iteration counts of IPOPT and the one-phase algorithm (Figure~\ref{fig:inf_cutest_iter_ratios}). 

\begin{figure}[H]
\includegraphics[scale=0.5]{\figDir{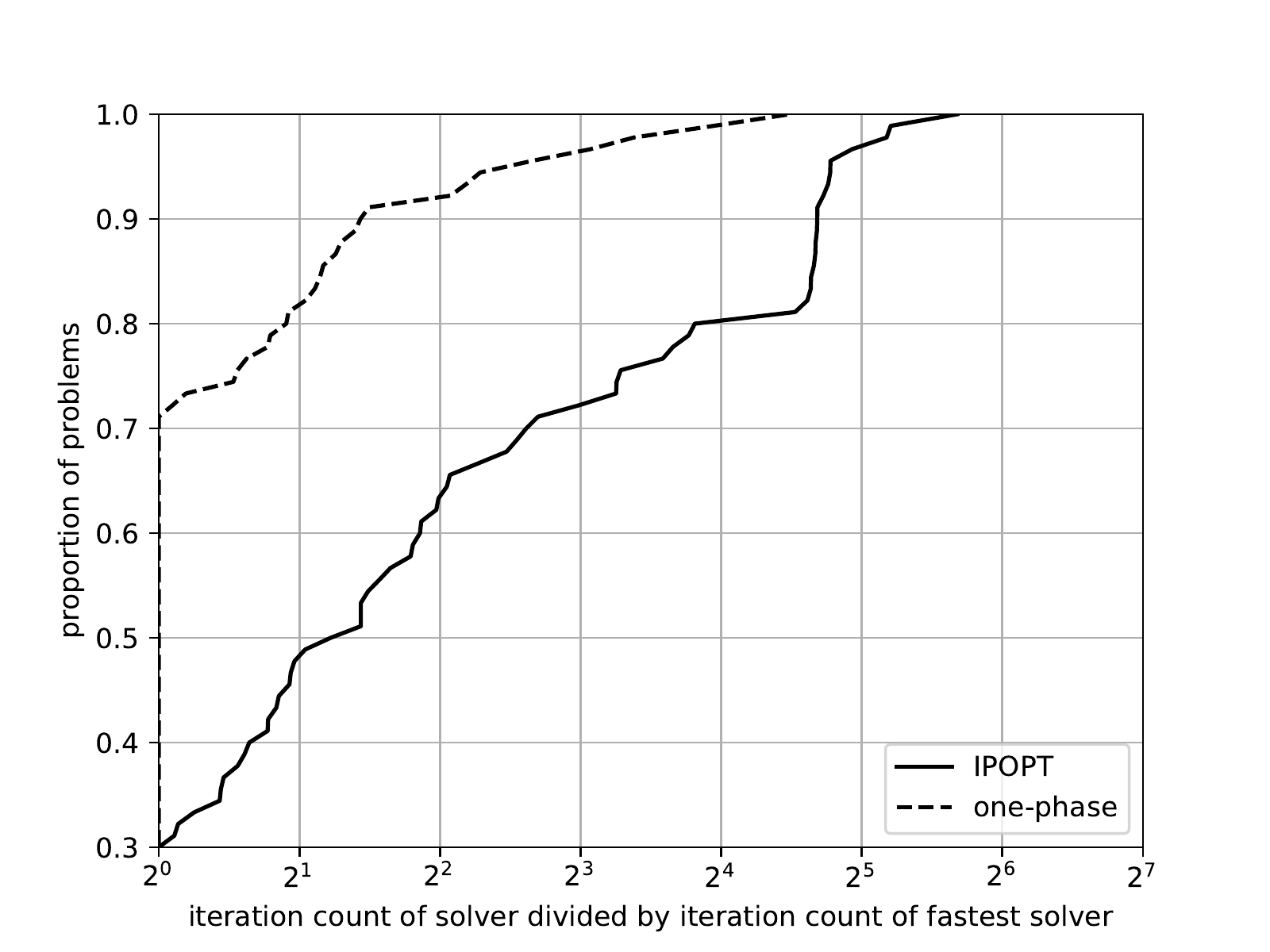}}
\caption{Performance profile on perturbed CUTEst problems where both algorithms declare the problem infeasible.}\label{fig:inf_cutest_iter_ratios}
\end{figure}

Since this perturbed CUTEst test set was somewhat artificial we also tested on the NETLIB test set containing infeasible linear programs\footnote{The LP problem CPLEX2 is an almost feasible problem that was declared feasible by both solvers. Therefore we removed it from our tests.} (Figure~\ref{fig:inf_netlib_iter_ratios}). We remark that on these 28 problems the one-phase algorithm required less iterations on every problem.

\begin{figure}[H]
\includegraphics[scale=0.5]{\figDir{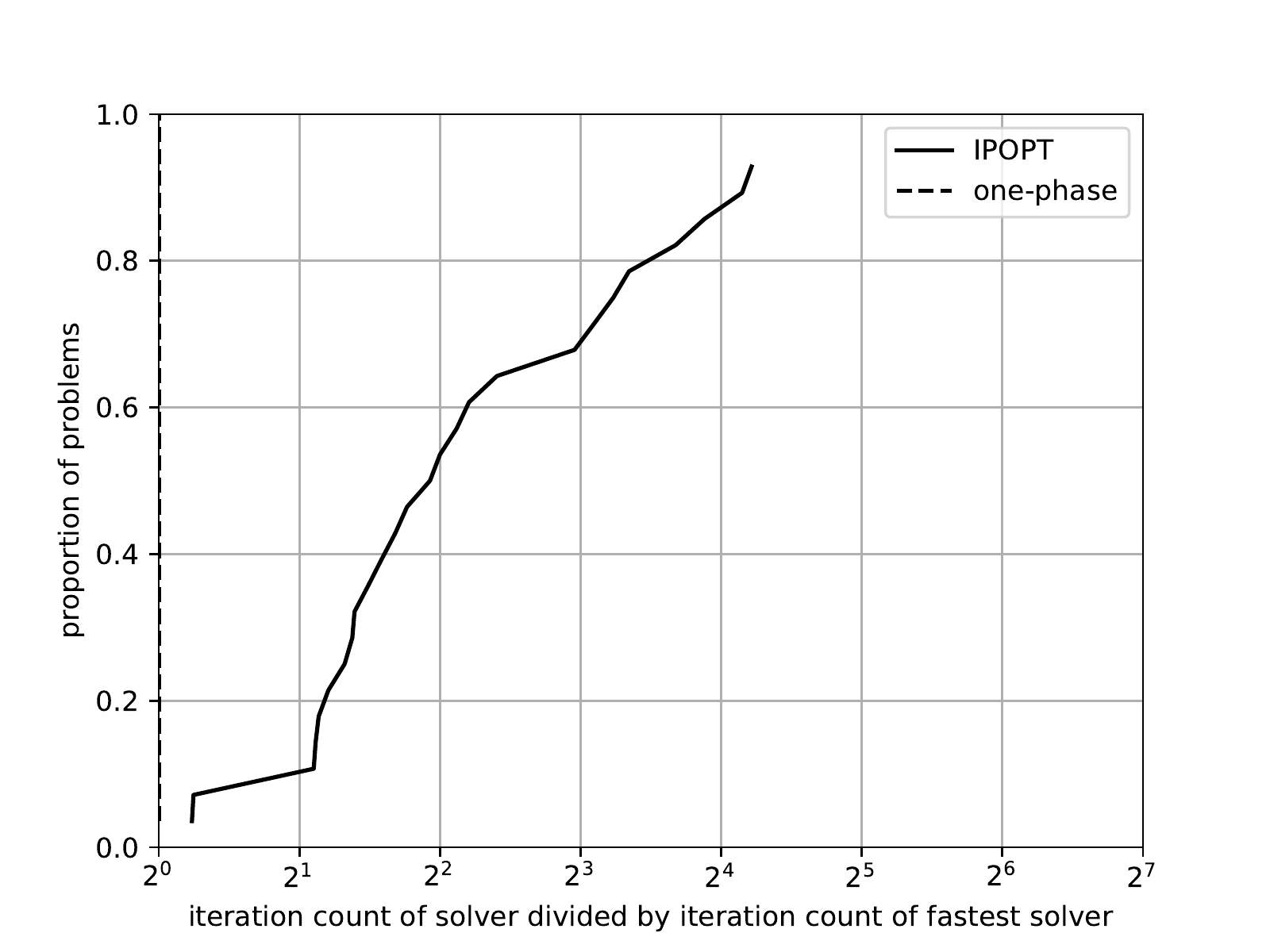}}
\caption{Performance profile on NETLIB infeasible test set.}\label{fig:inf_netlib_iter_ratios}
\end{figure}

\section{Conclusion and avenues for improvement}

This paper proposed a one-phase algorithm. It avoids a two phase or penalty method typically used in IPMs for nonlinear programming. Nonetheless, under mild assumptions it is guaranteed to converge to a first-order certificate of infeasibility, unboundedness or optimality. As we have demonstrated on large-scale test problems the algorithm has similar iteration counts to IPOPT, but significantly better performance on infeasible problems. 

An additional benefit of our approach is the ability to choose $\conWeight$. For example, if one has a starting point $x^0$ that strictly satisfies a subset of the constraints one can initialize the algorithm with $\conWeight_i = 0$ on this subset (we do this automatically for the bound constraints). The algorithm will then satisfy these constraints for all subsequent iterations. Aside from potentially speeding up convergence, this is particularly beneficial if some of the constraints or objective are undefined outside the region defined by this subset of the constraints.

One can interpret the relative value of $\mu^0$ and $\conWeight$ as the extent to which feasibility is prioritized over optimality: a larger $\mu^0$ gives more priority to feasibility. For a fixed $\conWeight$ picking a huge $\mu$ makes the IPM method behave like a phase-one, phase-two method: initially the algorithm attempts to find a feasible solution then it minimizes the objective on the feasible region. We have attempted such an initialization strategy and note that while it performs well on many problems, in others it causes the dual multipliers to become unduly large (as the theory of \cite{haeser2017behavior} predicts when, for example, the feasible region lacks an interior). We find that manually tuning the value of $\mu^0$ for a specific problem often significantly reduces the number of iterations. This sensitivity to the initialization, especially compared with the homogenous self-dual, is a known issue for Lustig's IPM for linear programming \cite[Table 1]{meszaros2015practical}. Therefore, we believe that improving our initialization scheme (Section~\ref{sec:initialization}) could result in significant improvements.

Finally, we note that improving the accuracy of the linear system solves would also improve our robustness. A disadvantage of using the Cholesky factorization is that we often had difficulty obtaining a sufficiently accurate solution as we approached optimality. Potentially switching to an \LBL{} factorization \cite{amestoy1998mumps,bunch1971direct} might help resolve this issue.

\section*{Acknowledgements}

We would like to thank Michael Saunders, Ron Estrin and Toshihiro Kosaki for useful conversations and feedback on the paper.

\bibliographystyle{abbrvnat} %
\bibliography{library-one-phase-2.bib}

\appendix

\section{Global convergence proofs for Algorithm~\ref{simple-one-phase}}\label{app:global-conv}

The purpose of this section is to provide proofs of supporting results for Theorem~\ref{thm:global-convergence}.

\subsection{Convergence of aggressive steps}

\subsubsection{Proof of Lemma~\ref{lemma:agg-succeeds}}\label{sec:lemma:agg-succeeds}

\lemAggSucceeds*
\begin{proof}
First, observe that as $\delta \rightarrow \infty$ the direction $\dir{x}$ computed from \eqref{eq:Schur-complement-system} tends to zero. Consider any $\alpha_{P} \in (0,1)$, since the function $a$ is continuous for sufficiently large $\delta$ we have
\begin{flalign}\label{eq:a-bound}
\| \cons(x) - \cons(x + \alpha_{P} \dir{x}) \|_{\infty} \le  \frac{\parCompAgg - \parComp}{2 \parCompAgg} \min_i\{ s_{i} \}
\end{flalign}
Now, set $\alpha_{P} = \minStepFunc(\mu, s)$ where $\minStepFunc$ is defined in \eqref{simple-min-step-size-aggresssive}, then for this choice of $\alpha_{P}$ we have
$$
\abs{ s^{+}_i - s_i } = \abs{  -\alpha_{P} \mu \conWeight_i +  \cons_i(x) - \cons_i(x + \alpha_{P} \dir{x}) } \le \frac{\parCompAgg - \parComp}{\parCompAgg} s_i,
$$
where the first equality holds by applying \eqref{eq:slackVarUpdate} and then \eqref{eq:muVarUpdate} with $\gamma = 0$, the second inequality equation \eqref{eq:a-bound} and $\alpha_{P} = \minStepFunc(\mu, s)$.
Note that 
$$
\frac{s^{+}_i y_i}{\mu} = \frac{s_i y_i}{\mu} (e + s^{-1}_i (s^{+}_i - s_i)) \in  \frac{s_i y_i}{\mu} \left[  \frac{\parComp}{\parCompAgg}, 1 + \frac{\parCompAgg - \parComp}{\parCompAgg} \right]  \subseteq \frac{s_i y_i}{\mu} \left[  \frac{\parComp}{\parCompAgg}, \frac{\parCompAgg}{\parComp} \right]  \subseteq [\parComp, 1/\parComp ]
$$
where the first transition comes from algebraic manipulation, the second transition follows from substituting our bound on $\abs{s^{+}_i - s_i}$, the third transition using $\frac{\parCompAgg - \parComp}{\parCompAgg} < \frac{\parCompAgg - \parComp}{\parComp}$, the final transition uses $ \frac{s_i y_i}{\mu} \in [\parCompAgg, 1/\parCompAgg]$. Therefore $\alpha_{D} = 0$ gives a feasible dual iterate. We conclude there exists a $\delta$ such that the $\alpha_{P}$ chosen by line~\ref{simple-agg-select-alpha-P} is at least $\minStepFunc(\mu, s)$, which proves the result.
\end{proof}

\subsubsection{Proof of Lemma~\ref{lem:yw-bounded}}\label{sub:lem:yw-bounded}

\lemYWbounded*

\begin{proof}
Since \eqref{terminate-primal-infeasible} does not hold: either $a(x)^T y \le 0$,
$\infeasFuncOne (\mu,x,s,y) > \TOLinfOne$ or $\infeasFuncTwo (\mu,x,s,y) > \TOLinfTwo$. We consider these three cases in order.

Let us consider the case that $a(x)^T y \le 0$ then $(\mu \conWeight - s)^T y = a(x)^T y \le 0$ by $\cons(\vec{x}) + \vec{s} = \mu \conWeight$. Re-arranging $(\mu \conWeight - s)^T y \le 0$ gives $y^T \conWeight \le s^T y / \mu \le m / \parComp$.

Let us consider the case that $\TOLinfOne < \infeasFuncOne (\mu,x,s,y) = \frac{\| \grad  a(x)^T y \|_{1}}{ a(x)^T y }$ then
\begin{flalign*}
w^T y &< s^T y + \frac{\| \grad  a(x)^T y \|_{1}}{ \TOLinfOne  } \le s^T y + \frac{\| \grad f(x) - \parConRegularizer \mu e^T \grad a(x) \|_{1} + \| \grad \Lag_{\mu}(x,y) \|_{1}}{ \TOLinfOne  } \\
&\le s^T y + 2 \frac{\| \grad f(x) - \parConRegularizer \mu e^T \grad a(x) \|_{1}}{ \TOLinfOne  } 
\end{flalign*}
where the first inequality holds by $\cons(\vec{x}) + \vec{s} = \mu \conWeight$ and re-arranging, the second by the triangle inequality and the third by the assumption that the aggressive step criteron \eqref{agg-criteron-farkas} is met. Furthermore, the term $s^T y$ is bounded by $\mu / \parCompAgg$ and the term $\| \grad f(x) - \parConRegularizer \mu e^T \grad a(x) \|_{1}$ is bounded because $f$ and $a$ are twice differentiable and the unboundness criterion~\eqref{terminate-dual-infeasible} is not met.

Let us consider the case that $\TOLinfTwo < \infeasFuncTwo (x,s,y)  = \frac{\| \grad  a(x)^T y \|_{1} + s^T y}{ \| y \|_{1} }$ then
\begin{flalign*}
\| y \|_{1} &<  \frac{s^T y + \| \grad  a(x)^T y \|_{1}}{\TOLinfTwo} \le \frac{ s^T y + \| \grad f(x) - \parConRegularizer \mu e^T \grad a(x) \|_{1} + \| \grad \Lag_{\mu}(x,y) \|_{1}}{\TOLinfTwo} \\
&\le 2 \frac{ s^T y + \| \grad f(x) - \parConRegularizer \mu e^T \grad a(x) \|_{1}}{\TOLinfTwo}
\end{flalign*}

where the first inequality holds by re-arranging, the second by the triangle inequality and the third by the assumption that the aggressive step criteron \eqref{agg-criteron-farkas} is met.
We conclude $\| y \|_{1}$ is bounded since $\| x \|$ is bounded, clearly therefore $w^T y$ is also bounded above.

Since in all three cases $\conWeight^T y$ is bounded we conclude the proof.
\end{proof}

\subsection{Convergence results for stabilization steps}

\subsubsection{Proof of Lemma~\ref{lem:compact-Q} and Corollary~\ref{coro:bound-everything}} \label{sec:lem:compact-Q}

We now introduce the set $\mathbb{Q}_{\mu, C}$ which we will use to represent the set of possible points the iterates of Algorithm~\ref{simple-one-phase} can take for a fixed $\mu$, i.e., during consecutive stabilization steps.

\begin{definition}\label{defQ}
Define the set $\mathbb{Q}_{\mu, C}$ for constants $\mu, C > 0$ as the set of points $(x,y,s) \in  \R^{\nvar} \times \R_{++}^{\ncon} \times \R_{++}^{\ncon}$ such that \eqref{restate:eq:barrier-primal-sequence-nice} holds and
\begin{enumerate}
\item \label{Q-phi-bounded-above} The function $\phi_{\mu}$ is bounded above, i.e., $\phi_{\mu}(x,y,s) \le C$.
\item \label{Q-bounded-below} The primal iterates are bounded, i.e., $\| x \| \le C$. 
\end{enumerate}
\end{definition}

Suppose Algorithm~\ref{simple-one-phase} generates consecutive stabilization steps stabilization steps $(\mu^k, x^k, s^k, y^k)$ for $k \in \{ \kStart, \dots, \kEnd \}$ with $\mu =\mu^{\kStart} = \dots = \mu^{\kEnd}$ and none of these iterates satisfy the unboundedness termination criterion~\eqref{terminate-dual-infeasible}. Let us show these iterates are contained in $\mathbb{Q}_{\mu, C}$, i.e., $(x^{k}, s^{k}, y^{k}) \in \mathbb{Q}_{\mu, C}$ for some $C > 0$. For $C \ge  \phi_{\mu}(x^{\kStart}, s^{\kStart}, y^{\kStart})$ condition \ref{defQ}.\ref{Q-phi-bounded-above} holds since during stabilization steps we only accept steps that decrease $\phi_{\mu}$. For sufficiently large $C$, condition \ref{defQ}.\ref{Q-bounded-below} holds from the assumption the unboundedness termination criterion~\eqref{terminate-dual-infeasible} is not satisfied.

We are now ready to prove Lemma~\ref{lem:compact-Q}. Note that during Lemma~\ref{lem:compact-Q} we will repeatedly use that the following elementary real analysis fact: 

\begin{fact}
Let $X = \{ x : g_i(x) \le 0 \}$. If $g_i$ is a continuous function and the set $X$ is bounded, then the set $X$ is compact.
\end{fact}

\begin{restatable}{lemma}{lemCompactQ}\label{lem:compact-Q}
Suppose assumptions~\ref{assume:diff} and \ref{assume:parameters} hold. 
For any constants $C, \mu > 0$ the set $\mathbb{Q}_{\mu, C}$ is compact.
\end{restatable}

\begin{proof}
First consider the set
$$
Q := \left\{ x \in \R^{\nvar} : (y, s) \in \R_{++}^{\ncon} \times \R_{++}^{\ncon}, \phi_{\mu}(x,y,s) \le C, \| x \| \le C \right\} 
$$
By $\| x \| \le C$ we see $Q$ is bounded. Furthermore, since $\phi_{\mu}(x) \le C$ and $Q$ is bounded there exists some constant $K_{1} > 0$ such that
$$
\mu w - \cons(x) \ge K_{1}
$$
for all $x \in Q$. Consider some sequence $x^{k} \in Q$ with $x^{k} \rightarrow x^{*}$. The statement $\cons(x) \le \mu w - K_{1}$ implies $\phi_{\mu}$ is continuous in a neighborhood of $x^{*}$. Using the definition of $Q$ and the assumption that $f$ and $a$ are continuous implies $x^{*} \in Q$, i.e., $Q$ is compact.

Note that:
$$
\mathbb{Q}_{\mu, C} = \left\{ (x,y,s) \in \R^{\nvar} \times \R_{++}^{\ncon} \times \R_{++}^{\ncon} : x \in Q, \cons(x) + s = \mu \conWeight, \frac{s_i y_i}{\mu} \in [\parComp, 1 / \parComp] ~~ \forall i \in \{1 , \dots, \ncon \} \right\}.
$$
Consider some $(x,y,s) \in \mathbb{Q}_{\mu, C}$, since $s = \mu w - \cons(x) \ge K_{2}$ and $\frac{s_i y_i}{\mu} \in [\parComp, 1 / \parComp]$ we can deduce $y$ is bounded. Since the function $\cons(x)$ and the term $s_i y_i$ are continuous we conclude $\mathbb{Q}_{\mu, C}$ is compact.
\end{proof}

\begin{restatable}{corollary}{coroBoundEverything}\label{coro:bound-everything}
Suppose assumptions~\ref{assume:diff} and \ref{assume:parameters} hold. Consider some fixed $\mu, C > 0$.
Then there exists some $L > 0$ such that for all $(x, s, y) \in \mathbb{Q}_{\mu, C}$ the following inequalities hold:
$$
s_i, y_i \ge 1/L
$$
$$
\| x \|, \| y \|, \| s \|, \| \grad \barrier_{\mu}(x) \|, \| \Schur \|, \| \grad \cons(x) \| \le L
$$
and for any $u \in \R^{\nvar}$ s.t. $\| u \| < 1 / L$
\begin{subequations}\label{lipschitz-continuous}
\begin{flalign}
\barrier_{\mu}(x + u) &\le \barrier_{\mu}(x) + \grad \barrier_{\mu}(x)^T u + L / 2 \| u \|^2 \label{phi-lipschitz-continuous} \\
\| \cons(x) + \grad \cons(x) u -\cons(x + u)  \| &\le L  \| u \|^2. \label{a-lipschitz-continuous-2nd}
\end{flalign}
\end{subequations}

Furthermore, if the aggressive criterion~\eqref{agg-criteron} does not holds then
$$
\max\{ \| \grad \barrier_{\mu}(x) \|, \| S y - \mu \ones \|_{\infty} \} \ge 1 / L.
$$
\end{restatable}

\begin{proof}
All these claims use Lemma~\ref{lem:compact-Q} and the elementary real analysis fact that for any continuous function $g$ on a compact set there $X$ there exists some $x^{*} \in X$ such that $g(x^{*}) = \sup_{x \in X}{g(x)}$.

The only non-trivial claim is showing \eqref{lipschitz-continuous}, which we proceed to show. Since there exists some constants $\varepsilon_{1} > 0$ such that $(x,y,s) \in \mathbb{Q}_{\mu, C}$ we have $\cons(x) \le \mu w - \varepsilon_{1}$. It follows that there exists some constant $\varepsilon_2 > 0$ such that for all $\| u \| \le \varepsilon_2$ we have $\cons(x + u) < \mu w$. Therefore there exists some $L > 0$ such that $\| \grad^2 \barrier_{\mu}(x) \| \le L$.

For some $x$ and $\nu$ with $\| \nu \| = 1$ define the one dimensional function
$$
h(\alpha) :=  \barrier_{\mu}(x + \alpha \nu) 
$$
then for $\alpha \in [0, \varepsilon_2]$ we get
$$
h(\alpha) - h(0) - \alpha h'(0) = \int_{0}^{\alpha}{ \int_{0}^{\eta_{2}}{h''(\eta) \partial \eta_{1} \partial \eta_{2}} } \le \alpha^2 L / 2,
$$
which using $u = \nu \alpha$ for $\alpha \in [0, \varepsilon_2]$ concludes the proof of \eqref{phi-lipschitz-continuous}. Showing \eqref{a-lipschitz-continuous-2nd} consists of a similar argument.
\end{proof}

\subsubsection{Proof of Lemma~\ref{lemConsecutiveStable}}\label{sec:lemConsecutiveStable}

Before starting the proof of  Lemma~\ref{lemConsecutiveStable}, we remark that Lemma~\ref{lemConsecutiveStable} uses Lemma~\ref{lem:compact-Q} and Corollary~\ref{coro:bound-everything} which are proved in Section~\ref{sec:lem:compact-Q}

\lemConsecutiveStable*

\begin{proof}
Recall that, as we discussed following Definition~\ref{defQ}, there exists some $C > 0$ such that for any consecutive series of stabilization steps $(\mu^k, x^k, s^k, y^k)$ for $k \in \{ \kStart, \dots, \kEnd \}$ we have $(x^k, s^k, y^k) \in \mathbb{Q}_{\mu, C}$ with $\mu = \mu_{k_{\kStart}} = \mu_{k_{\kEnd}}$.

From Corollary~\ref{coro:bound-everything} we know that there exists some $L > 0$ such that $\max\{-\lambda_{\min}(\Schur), \lambda_{\max}(\Schur) \} \le L$, therefore 
\begin{flalign*}
\dir{x}^T \grad \psi(x) = -\grad \psi(x)^T (\Schur + \delta I)^{-1} \grad \psi(x) \le -\| \grad \psi(x) \|^2 / (L + \delta) \le -\| \grad \psi(x) \|^2 / (3 L + \mu),
\end{flalign*}
where the first transition uses the definition of $\dir{x}$ in \eqref{eq:Schur-complement-system} with $\gamma = 1$ when computing $b$, the second transition $\lambda_{\max}(\Schur) \le L$ and the third transition uses $\delta \le 2 L + \mu$ from line~\ref{alg-simple-delta-min} of Algorithm~\ref{simple-one-phase}.

Similarly, by line~\ref{alg-simple-delta-min} of Algorithm~\ref{simple-one-phase} we get $\Schur + \delta \eye \succeq \mu \eye$ therefore
\begin{flalign*}
\| \dir{x} \|^2 = \| (\Schur + \delta I)^{-1} \grad \psi(x)\|^2 \le \| \grad \psi(x)\|^2 / \mu^2.
\end{flalign*}

Furthermore,
\begin{flalign*}
\barrier_{\mu}(x^{+}) - \barrier_{\mu}(x)  &\le \alpha_{P} \grad \barrier_{\mu}(x)^T \dir{x} + \frac{L \alpha_{P}^2}{2} \| \dir{x} \|^2 \\
&= \alpha_{P} \grad \barrier_{\mu}(x)^T \dir{x} + \left( \frac{L \alpha_{P}^2}{2} + \parObjReductFactor \frac{\alpha_{P}  \delta}{2} \right) \| \dir{x} \|^2  - \parObjReductFactor  \frac{\alpha_{P}  \delta}{2} \| \dir{x} \|^2  \\
&\le \alpha_{P} \left( \parObjReductFactor \grad \barrier_{\mu}(x)^T \dir{x} + \| \grad \barrier_{\mu}(x)\|^2 \left( \frac{1}{\mu^2} \left(\frac{L \alpha_{P}^2}{2} + \parObjReductFactor \frac{\alpha_{P}  \delta}{2} \right) - \frac{1 - \parObjReductFactor}{3 L + \mu} \right)  \right) - \parObjReductFactor  \frac{\alpha_{P}  \delta}{2} \| \dir{x} \|^2 \\
&\le \parObjReductFactor \alpha_{P} \left( \grad \barrier_{\mu}(x)^T \dir{x} - \frac{\alpha_{P}  \delta}{2} \| \dir{x} \|^2 - c_{1} \| \grad \barrier_{\mu}(x)\|^2  \right) 
\end{flalign*}
where the first transition holds by Corollary~\ref{coro:bound-everything}, the second by adding and subtracting terms, the third by the above inequalities, the fourth for some constant $c_{1} > 0$ with $\alpha_{P}$ sufficiently small, i.e., any $\alpha_{P} \in (0, c_{1})$.

We can bound $\| \dir{x} \|$, $\| \dir{y} \|$ and $\| \dir{s} \|$ using our bound on $\| \dir{x} \|$, Corollary~\ref{coro:bound-everything} and \eqref{compute-ds-dy}.

Finally,
\begin{flalign*}
\| S^{+} y^{+} - \mu \|_{\infty}  &= \| (s + \dir{s}) y^{+} - \mu + (s^{+} - s - \dir{s}) y^{+} \|_{\infty} \\
&\le \| S y - \mu \ones \|_{\infty} + \alpha_{P} \left( -\| S y - \mu \ones \|_{\infty} + \alpha_{P} L \| \dir{x} \|^2 \| y^{+} \|_{\infty} + \alpha_{P} \| \dir{s} \|_{\infty} \| \dir{y} \|_{\infty} \right) \\
&\le  \| S y - \mu \ones \|_{\infty} (1 - \alpha_{P}  + c_{2} \alpha_{P}^2 )
\end{flalign*}
where the second transition holds by $S y + S \dir{y} + Y \dir{s} = \mu$ and Corollary~\ref{coro:bound-everything} which shows $\| s + \dir{s} - s^{+}  \|_{\infty} \le L \| \dir{x} \|^2$, the second inequality by the fact that the directions and $\| y \|$ are bounded.

Using $\MeritComp_{\mu}(s,y) := \frac{\| S y - \mu \ones \|_{\infty}^3}{\mu^2}$, the previous expression and the boundedness of $s$ and $y$ we get
$$
\MeritComp_{\mu}(s^{+},y^{+}) \le \MeritComp_{\mu}(s,y) (1 - \alpha_{P}) + \bar{c}_{2} \alpha_{P}^2
$$
for some constant $\bar{c}_{2} > 0$. Defining
$$
\Upsilon(\alpha_{P}) := \alpha_{P} \parObjReductFactor \left( \frac{1}{2} \left( \grad \psi_{\mu}(x)^T  \dir{x} - \frac{\delta}{2} \alpha_{P} \norm{ \dir{x}}^2 \right) -  \MeritComp_{\mu}(s,y)  \right),
$$
we get for some constants $c_{1}, c_{2}, c_{3} > 0$ that
\begin{flalign*}
\phi_{\mu}(x^{+},y^{+},s^{+}) - \phi_{\mu}(x,y,s) &\le \Upsilon(\alpha_{P}) + \alpha_{P} \left( c_{2} \alpha_{P}  - c_{3} \MeritComp_{\mu}(s,y) - c_{1} \| \grad \barrier_{\mu}(x)\|^2 \right)
\end{flalign*}
by using $\phi_{\mu}(x,y,s) := \psi_{\mu}(x) + \MeritComp_{\mu}(s,y)$ and substituting our upper bounds on $\barrier_{\mu}(x^{+}) - \barrier_{\mu}(x)$ and $\MeritComp_{\mu}(s^{+},y^{+})$.

Since $\max\{ \MeritComp_{\mu}(s,y), \| \grad \barrier_{\mu}(x)\| \}$ is bounded away from zero, we deduce the largest $\alpha_{P}$ satisfying $\phi_{\mu}(x^{+},y^{+},s^{+}) - \phi_{\mu}(x,y,s) \le \Upsilon(\alpha_{P})$ is bounded away from zero and below by zero. We conclude that we must reduce $\phi_{\mu}$ by a constant amount each iteration, which means that if there is an infinite sequence of stabilization steps then $\phi_{\mu}(x^k, s^k, y^k) \rightarrow -\infty$ and hence $\| x^k \| \rightarrow \infty$.
\end{proof}

\section{Further implementation details}

\subsection{Matrix factorization strategy}\label{sec:mat-fact}

This strategy is based on the ideas of IPOPT \cite[Algorithm IC]{wachter2006implementation}.

\begin{algorithm}[H]
\textbf{Input:} The matrix $\Schur$ and previous delta choice $\delta \ge 0$ \\
\textbf{Output:} The factorization of $\Schur +  \delta \eye$ for some $\delta > 0$ such that the matrix $\Schur +  \delta \eye$ is positive definite.
\begin{enumerate}[label*=A.{\arabic*}]
\item Set $\deltaPrev \gets \delta$.
\item Compute $\tau \gets \min_i \Schur_{i,i}$. If $\tau \le 0$ go to line~\ref{line:new-delta}.
\item Set $\delta \gets 0$, $\tau \gets 0$.
\item Perform Cholesky factorization of $\Schur$, if the factorization succeeds then return factorization of $\Schur$.
\item\label{line:new-delta} Set $\delta \gets \max\{ \deltaPrev \parDeltaDecrease, \parDeltaMin - \tau \}$.
\item If $\delta \ge \parDeltaMax$ then terminate the algorithm with $\status = \failure$.
\item Perform Cholesky factorization of $\Schur + \delta \eye$,  if the factorization succeeds then return factorization of $\Schur + \delta \eye$.
\item Set $\delta \gets \parDeltaIncreaseFailure \delta$. Go to previous step.
\end{enumerate}
\caption{Matrix factorization strategy}\label{alg:mat-fact}
\end{algorithm}
where $\parDeltaDecrease, \parDeltaIncreaseFailure, \parDeltaMin, \parDeltaMax$ have default values of $\parDeltaDecreaseValue, \parDeltaIncreaseFailureValue, \parDeltaMinValue, \parDeltaMaxValue$ respectively.

\subsection{Initialization}\label{sec:initialization}

This section explains how to select initial variable values $(\mu^0,x^0, y^0, s^0)$ to pass to Algorithm~\ref{practical-one-phase-IPM}, given a suggested starting point $x_{\text{start}}$.

The first goal is to modify $x_{\text{start}}$ to satisfy any bounds e.g., $l \le x \le u$. Let $\mathcal{B} \subseteq \{ 1, \dots, \ncon \}$ be the set of indices corresponding to variable bounds. More precisely, $i \in \mathcal{B}$ if and only if there exists $c \in \R$ and $j \in \{ 1, \dots, \nvar \}$ such that $a_i(x) = x_j + c$ or $a_i(x) = -x_j + c$. We project $x_{\text{start}}$ onto the variable bounds in the same way as IPOPT \cite[Section 3.7]{wachter2006implementation}. Furthermore, we set $w_i = 0$ and $s_i^{0} = -a_i(x^0)$ for each $i \in \mathcal{B}$. This ensures that the variable bounds are satisfied throughout, and is useful because the constraints or objective may not be defined outside the bound constraints. To guarantee that any constraint $a_i(x)$ that was strictly feasible at the initial point $x^{0}$ remains feasible throughout the algorithm we could simply set $\conWeight_i = 0$ and $s_i^{0} = -a_i(x^0)$.
 
The remainder of the initialization scheme is inspired by Mehrotra's scheme for linear programming \cite[Section 7]{mehrotra1992implementation} and the scheme of \citet*{gertz2004starting} for nonlinear programming. %
Set
\begin{flalign*}
\tilde{y} &\gets \ones \\
\tilde{s} &\gets -\cons(x^{0}) + \max\{  -2 \min_i\{ s_i \}, \parInitialize \}
\end{flalign*}
for some parameter $\parInitialize \in (0,\infty)$ with default value $\parInitializeValue$. Then factorize $\Schur + I \delta$ as per lines~\ref{line:init-delta} to \ref{line:factor-schur} of Algorithm~\ref{practical-one-phase-IPM}. Find directions $(d_{x}, d_{y}, d_{s})$ via \eqref{practical-direction} and set $\tilde{y} \gets \tilde{y} + d_y$ and $\tilde{s} = - \cons(x^{0})$.
Next, we set:
\begin{flalign*}
 \varepsilon_{y} &\gets \max\{  -2 \min_i{ \tilde{y}_i}, 0 \}  \\
\tilde{y} &\gets \tilde{y} + \varepsilon_{y} \\
 \varepsilon_{s} &\gets \max\left\{  - 2 \min_i{ \tilde{s}_i}, \frac{ \| \grad \Lag_{0}(x^0,\tilde{y}) \|_{\infty} }{\| \tilde{y} \| + 1} \right\} \\
\tilde{s}_i &\gets \tilde{s}_i + \varepsilon_{s} \quad \forall i \not\in \mathcal{B},  
\end{flalign*}
and then
\begin{flalign*}
\tilde{y} &\gets \tilde{y} +  \frac{\tilde{s}^T \tilde{y}}{2 e^T \tilde{s}} \\
\tilde{y} &\gets \max\{  \parInitializeMin, \min\{ \tilde{y},  \parInitializeMax \} \}  \\
\tilde{s}_i &\gets \tilde{s}_i + \frac{\tilde{s}^T \tilde{y}}{2 e^T \tilde{y}} \quad \forall i \not\in \mathcal{B} \\
\tilde{\mu} &\gets  \frac{\tilde{s}^T \tilde{y}}{\ncon}.
\end{flalign*}
where $\parInitializeMin \in \parInitializeMinInterval$ has a default value of $\parInitializeMinValue$ and $\parInitializeMax \in \parInitializeMaxInterval$ has a default value of $\parInitializeMaxValue$.

Next, set
\begin{flalign*}
\mu^0 &\gets \parMuScale \tilde{\mu} \\
s^0 &\gets \tilde{s} \\
\conWeight &\gets \frac{a(x^0) + s^0}{\mu^{0}}
\end{flalign*}
where $\parMuScale \in \parMuScaleInterval$ is a parameter with a default value of 1.0. We leave $\parMuScale$ as a parameter for the user because we notice that for some problems changing this value can reduce the iteration count by an order of magnitude. Devising a better way to select $\mu^0$ we believe could significantly improve our algorithm. Finally, we need to ensure the dual variables satisfy \eqref{eq:comp-slack} so we set
\begin{flalign*}
y^0 &\gets \min \{ \max \{  \parCompAgg \mu^0 (S^0)^{-1} e, \tilde{y} \},  \mu / \parCompAgg (S^0)^{-1} e \}.
\end{flalign*}

\section{Details for IPOPT experiments}

We used the following options with IPOPT:

\begin{enumerate}
\item tol: $10^{-6}$
\item max\_iter: $3000$
\item max\_cpu\_time: $3600$
\item nlp\_scaling\_method: `none'
\item bound\_relax\_factor: $0.0$
\item acceptable\_iter: $999999$
\end{enumerate}
All other options remained at their default values.

\end{document}